\documentclass[12pt]{article}
\usepackage{graphicx}

\usepackage{amssymb,amsmath,amsthm}
\usepackage{mathrsfs}
\usepackage{bbold}
\usepackage{enumitem}
\usepackage{bm}
\usepackage{etoolbox}
\usepackage{titlesec}
\newcommand{\subscript}[2]{$#1 _ #2$}
\usepackage[colorlinks=true]{hyperref}
\usepackage{color}
\newcommand{\blue}{}
\newcommand{\black}{}

\def\saf{s\alpha_f}
\def\of{\omega_f}
\def\epsilon{\varepsilon}

\def\bC{{\mathbb{C}}}
\def\bR{{\mathbb{R}}}

\def\CPo{{\mathbb{CP}^1}}
\def\RPo{{\mathbb{RP}^1}}

\def\cK{{\mathcal K}}

\def\cR{{\mathcal R}}

\def\cT{{\mathcal T}}

\def\Jall{J^{int}}

\newcommand{\bZ}{{\mathbb Z}}
\newcommand{\bN}{{\mathbb N}}

\newtheorem*{theorem*}{Theorem}
\newtheorem{thm}{Theorem}
\newtheorem{theoremX}{Theorem}

\newtheorem*{conjecture*}{Conjecture}
\newtheorem{defn}{Definition}
\newtheorem{esempio}{Example}
\newtheorem{prop}{Proposition}
\newtheorem{rmk}{Remark}
\newtheorem{cor}{Corollary}
\newtheorem{lm}{Lemma}

\begin{document}

\markboth{Roberto De Leo}{Backward asymptotics in S-unimodal maps}

\title{%
  Backward asymptotics in S-unimodal maps.
}

\author{Roberto De Leo\\ \small Howard University, Washington DC 20059 (USA)\\ \small INFN, Cagliari (Italy)}

\maketitle
%
\vspace{0.5cm}
\centerline{\em To Jim Yorke, on the occasion of his 80th brithday.}
\vspace{0.5cm}
\begin{abstract}
  While the forward trajectory of a point in a discrete dynamical system is always unique, in general a point can have infinitely many
  backward trajectories. The union of the limit points of all backward trajectories through $x$ was called by M.~Hero the
  ``special $\alpha$-limit'' ($s\alpha$-limit for short) of $x$. In this article we show that there is a hierarchy of $s\alpha$-limits
  of points under iterations of a S-unimodal map: the size of the $s\alpha$-limit of a point increases monotonically as the point gets
  closer and closer to the attractor. The $s\alpha$-limit of any point of the attractor is the whole non-wandering set.
  This hierarchy reflects the structure of the graph of a S-unimodal map recently introduced jointly by Jim Yorke and the present author.
\end{abstract}  

\section{Introduction}
Our interest in $s\alpha$-limits comes from their relation to the edges of the graph of a dynamical system,
that we studied in a recent joint work with J.~Yorke~\cite{DLY20} in case of 
S-unimodal maps. What we ultimately show in~\cite{DLY20} is that, in the 1-dimensional real case, there is a hierarchy among
repellors so that points arbitrarily close to a given repellor have trajectories that asymptote to all repellors below him in this
hierarchy and to no other repellor. In~\cite{DLY21} we present numerical evidence that this phenomenon is common also in multidimensional
real discrete dynamical systems.

Due to this hierarchy, backward limits are much more complicated than in the 1-dimensional {\em complex} case,
where there is, instead, a unique repellor, the Julia set, and this repellor is both backward and forward invariant.
We believe that this is the main reason why the standard definition of $\alpha$-limit was never questioned
by the discrete complex dynamics community.
In this article, we argue that the concept of $s\alpha$-limit is a more suitable analogue of $\omega$-limit and
use our results in~\cite{DLY20} to find the $s\alpha$-limit of points of the interval under a S-unimodal map.

The idea of encoding the qualitative behavior of a dynamical system in a graph goes back to S.~Smale that, in~\cite{Sma67},
proved that the non-wandering set $\Omega_f$ of an Axiom-A diffeomorphism $f$ of a compact manifold decomposes in a unique way
as the finite union of disjoint, closed, invariant indecomposable subsets $\Omega_i$, on each of which $f$ is topologically transitive.
Hence one can associate to $f$ a directed graph, describing its qualitative dynamics, whose nodes are these $\Omega_i$ and such that
there is an edge from $\Omega_j$ to $\Omega_i$ if the stable manifold of $\Omega_i$ intersects the unstable manifold of $\Omega_j$.

A decade later C.~Conley, in~\cite{Con78}, widely generalized this idea to continuous finite-dimensional dynamical systems replacing
the non-wandering set with the chain-recurrent set. One of the main advantages of using the chain-recurrent set is
that on this set there is a natural equivalence relation, that does not extend in general to the non-wandering set,
whose equivalence classes turns out to be the exact analogue of the indecomposable subsets of Axiom-A diffeomorphisms.
Conley's construction was later generalized to discrete dynamics by D.~Norton~\cite{Nor95b} and to several other settings
by other authors (see~\cite{DLY21} for a detailed bibliography on the subject).

The chain-recurrent set $\cR_f$ 
of a dynamical system $f$ on $X$ is defined via $\epsilon$-chains, namely finite sequences
of points each of which is within $\epsilon$ from the image of the previous one under $f$ (see the next section for more precise definitions
about the concepts used in this introduction).
A point $x\in X$ is {\em upstream} from a point $y\in X$ if there is an $\epsilon$-chain that starts at $x$ and ends at $y$ for every $\epsilon>0$.
We also say that $y$ is {\em downstream} from $x$.
A point is chain-recurrent if it is upstream (or, equivalently, downstream) from itself.

The relation
$$
x\sim y\hbox{ if }x\hbox{ is both upstream and downstream from }y
$$
is an equivalence relation. We call nodes its
equivalence classes and we associate to $f$ the graph $\Gamma_f$ having the nodes of $\cR_f$ as its nodes and so that
there is an edge from node $A$ to node $B$ if there is a bi-infinite trajectory (bitrajectory)
$$t=(\dots,t_{-2},t_{-1},t_0,t_1,t_2,\dots),$$
whose limit points for $n\to-\infty$ lie in $A$ and those for $n\to+\infty$ lie in $B$.
The dynamical relevance of $\Gamma_f$ comes from Conley's decomposition theorem:
given a continuous map $f$ of a compact metric space $X$ into itself, a point $x\in X$ either belongs to a node
of $\Gamma_f$, and so it has some kind or ``recurrent'' dynamics, or is ``gradient-like'', namely it defines an edge
between two distinct nodes of $\Gamma_f$ (see~\cite{Nor95b}, Thm.~4.10).

In~\cite{DLY20} we studied the structure of the graph of a S-unimodal map and proved that such graph is always a tower,
namely that there is an edge between each pair of nodes.
Since there cannot be loops in the graph of a dynamical system, this means that there is a
linear hierarchy among all nodes: there is a first node $N_0$ that has edges towards each other node, a node $N_1$ that has edges
towards each other node except $N_0$ and so on up to the attracting node $N_p$. Each node but one is repelling and 
every repelling node, which can be either a cycle (for short we call {\em cycle} a periodic orbit) or a Cantor set,
has an edge towards the attracting node. 

Figure~\ref{fig:bd} shows the nodes of the logistic map $\ell_\mu(x)=\mu x(1-x)$, $\mu\in[2.9,4]$, that are visible
at the picture's resolution.
The attracting nodes are painted in shades of gray, representing the density of a generic orbit of the attractor.
There are five kinds of attracting nodes: type $A_1$, a cycle; type $A_2$, a cycle of intervals; type $A_3$, an adding machine;
type $A_4$, a 1-sided attracting cycle belonging to a repelling Cantor set (taking place at the beginning of each window);
type $A_5$ (see Thm.~\ref{thm:AttractingNodes} about this complicated case, taking place at the end of each window). 
The repelling nodes are painted either in green (if they are cycle nodes) or in red (if they are Cantor set nodes).
The picture also shows the curves $c_k=\ell_\mu^k(c)$, $k=1,\dots,4$, where $c=1/2$ is the critical point of the logistic map.
The points $c_k$ play a fundamental role in the forward and backward dynamics of S-unimodal maps.
\begin{figure}
 \centering
 \includegraphics[width=18cm]{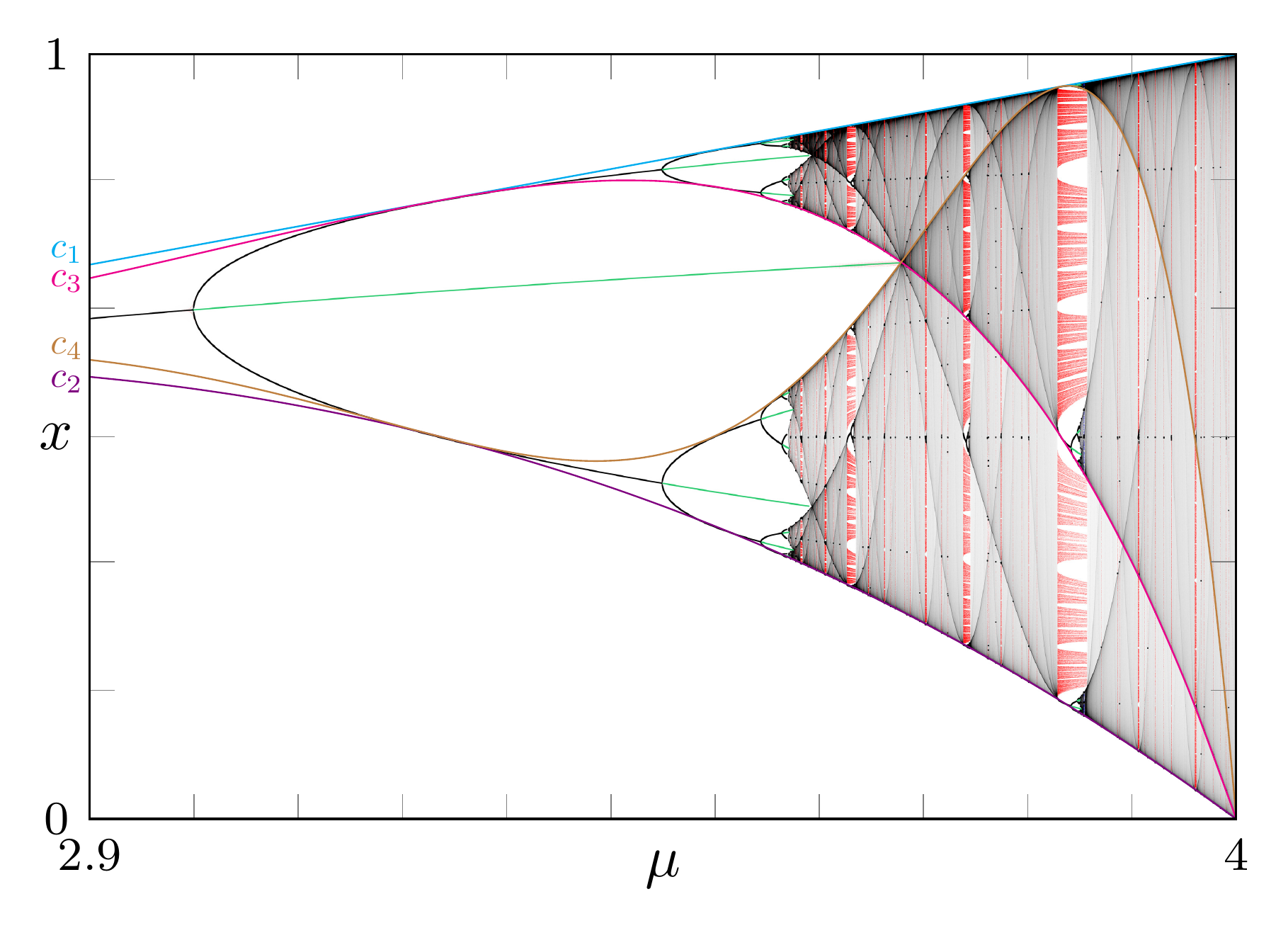}
  \caption{{\bf Bifurcations diagram of the logistic map.}
    This picture shows the nodes of the logistic map $\ell_\mu(x)=\mu x(1-x)$ in the parameter range $[2.9,4]$ that are visible at this resolution.
    Attractors are painted in shades of gray (depending on the density), repelling periodic orbits in green and repelling Cantor sets in red.
    The colored lines labeled by $c_k$ are the lines $\ell_\mu^k(c)$, where $c=0.5$ is the critical point of $\ell_\mu$.
    The black dots that are visible within the diagram are points belonging to low-period cycles, signaling the presence of bifurcation cascades
    within some window. The fact that some of them keep close to the line $x=c$ is a reflection of Singer's Theorem: when the attractor
    is a periodic orbit, $c$ belongs to its immediate basin.
    Notice that, for each $\mu$, all nodes but $N_0=0$ lie within the interval $[c_2,c_1]$ since points in $(c_1,1]$ are outside the range of the map
    and, for $\mu\geq2$, all points in $(0,c_2)$ eventually fall in $[c_2,c_1]$, which is forward invariant.
    See Sec.~\ref{sec:bd} for further comments of this figure.
  }
  \label{fig:bd}
\end{figure}

The main goal of the present article is to use our results in~\cite{DLY20} to study the structure of the $s\alpha$-limit sets
under S-unimodal maps. 
After presenting, in Sections~2 and~3, all definitions and main results we need from literature, we study in Section~4 the
backward dynamics of points under a S-unimodal map $f:[0,1]\to[0,1]$, $f(0)=0$. 
\blue 
Except when $N_p$ is of type $A_5$  (see Thm.~\ref{thm:main} for a complete statement),
our main result can be formulated as follows.
When $f$ has $p$ nodes $N_0,\dots,N_p$, $p\geq1$, then
\black
$[0,1]$ can be written as the disjoint union of sets $U_{-1},U_0,\dots,U_p$ such that:
\begin{enumerate}
\item each $U_k$, {\blue $k<p$}, is the finite union of one or more intervals (these intervals can be open, half-open or closed);
\item $U_{-1}=(c_2,1]$ is the set of all points that have no bitrajectory passing through them;
\item $U_0=[0,c_2)$ is the set of all points with a single bitrajectory, which necessarily asymptotes backward to $0$;
\item $N_k\subset U_k$ for each $k=0,\dots,p$;
\item {\blue $U_p=N_p$};
\item {\blue $s\alpha_f(x)=\cup_{i=0}^kN_i$ for each $x\in U_k$, $k=0,\dots,p$;}
\end{enumerate}
%
%
Notice that proving this theorem requires different techniques for points that are not chain-recurrent and points
that are so and, for chain-recurrent points, it requires different techniques for different types of nodes.
%

Figs.~\ref{fig:U1}--\ref{fig:U3} show examples of these $U_k$ sets in several cases for the logistic map. For instance,
let us describe $U_1$ in several cases when $\ell_\mu$ has three nodes.
When $\mu$ is close enough to 3 from the right, $N_1=\{\bar p_\mu\}$
is the non-zero fixed point of $\ell_\mu$ and $N_2=\{p_1,p_2\}$, $p_1<p_2$, a stable 2-cycle. In this case,
$$
U_1 = [c_2,p_1)\cup(p_1,p_2)\cup(p_2,c_1].
$$

When $\mu$ is a chaotic parameter with just three nodes and close enough from the left to $\mu_M\simeq3.679$,
which is the first parameter value at which $c_3$ and $c_4$ meet in Fig.~\ref{fig:bd}, then $N_1=\{\bar p_\mu\}$
is again the non-zero fixed point of $\ell_\mu$ but this time $N_2=[c_2,c_4]\cup[c_3,c_1]$. Hence, in this case,
$$
U_1 = (c_4,c_3).
$$

When $\mu$ is close enough from the right to $\mu_0=1+2\sqrt{2}$, which is the left endpoint of the period-3 window
(see Figs.~\ref{fig:bd} and~\ref{fig:T2}), then $N_1$ is a Cantor set (painted in red) and $N_2=\{p_1,p_2,p_3\}$, $p_3<p_1<p_2$,
is a stable 3-cycle (painted in black). Hence
$$
U_1 = [c_2,p_3)\cup(p_3,p_1)\cup(p_1,p_2)\cup(p_2,c_1].
$$



Our main result above achieves also two sub-goals. First, it proves for S-unimodal maps Conjecture~1 (see Sec.~2) by Kolyada, Misiurewicz
and Snoha, namely that $s\alpha$ limits are closed sets for all $x\in[0,1]$.
Second, it supports several numerical observations by the author in planar discrete dynamical systems~\cite{DL20} suggesting that,
in case of real Newton maps on the plane, besides a non-trivial open set of points with backward trajectories that asymptote to
(almost) the whole Julia set, just as it happens in the complex case (Theorem~\ref{thm:saC}), there are also non-trivial open sets
from which one get smaller subsets of the non-wandering set or even the empty set.
\section{Forward and Backward limits}
The study of (time-)discrete dynamical systems, namely iterations of a continuous map $f:X\to X$ on a metric space $(X,d)$,
goes back to the introduction of Poincar\'e maps~\cite{Poi90,Poi99} as a tool to study the qualitative
behavior of the integral trajectories of a vector field. One of the main goals of the field is finding the asymptotics of points under a
given map $f$: where a point $x\in X$ ultimately ends up (the ``$\omega$-limit'' of $x$)
and where it might come from (the ``$\alpha$-limit'' of $x$).

\subsection{Forward asymptotics.}
The asymptotics in the forward direction present no ambiguity: given a point $x\in X$, the set
$$
\omega_f(x)=\bigcap_{n\geq0}\;\overline{\bigcup_{m\geq n}\{f^m(x)\}}
$$
consists in the accumulation points of the forward orbit of $x$, namely the sequence $\{f(x),f^2(x),f^3(x),\dots\}$ where,
as usual, we use the notation
$$
f^n(x)=f(f^{n-1}(x)).
$$
The setting where discrete dynamics has been possibly more successful to date is
iterations of holomorphic (and therefore rational) maps on the Riemann sphere $\CPo$. In this setting, $F_f$ (Fatou set)
denotes the largest open set over which the family of iterates $\{f^n\}$ is normal and $J_f$ (Julia set) its complement
and we have the fundamental full classification below.
\blue
\begin{defn}
  We use {\bf cycle} as a synonim for periodic orbit. By $n$-cycle we mean a periodic orbit of period $n$
  (so, in particular, a $n$-cycle has $n$ elements). We denote by $|\gamma|$ the period of a cycle $\gamma$.
\end{defn}
\black
\begin{theoremX}[{\bf $\bm\omega$-limits in $\bm\bC$} \cite{Sul85}]
  \label{thm:oC}
  Let $f:\CPo\to\CPo$ be a rational function of degree larger than 1. Then:
  \begin{enumerate}
    \item if $x\in J_f $, then $\omega_f(x)$ is either a repelling cycle or the whole $J_f$;
    \item if $x\in F_f$, then $\omega_f(x)$ is either a hyperbolic or parabolic attracting cycle or a Jordan curve.
  \end{enumerate}
\end{theoremX}
Recall that, in the complex setting, the duality {\em non-chaotic} versus {\em chaotic} dynamics coincides with
the duality {\em attracting}  versus {\em repelling} set. In the real case, except in particular cases such as
Newton maps (e.g. see~\cite{DL18}), the two dualities are independent (for instance, there are chaotic attractors).
In regard with orbits asymptotics, the latter is the most relevant. An analogue of Theorem~\ref{thm:oC}
for the real case is given by the following weaker result:
%
\begin{theoremX}[{\bf $\bm\omega$-limits in $\bm\bR$}~\cite{MMS92}]
  \label{thm:oR}
  Let $f$ be either a generic non-invertible $C^2$ self-map of $\RPo$ or a S-multimodal map
  on the interval. Then, for almost all
  $x$, $\omega_f(x)$ can be  of the following three types:
  \begin{enumerate}
    \item a cycle;
    \item a minimal Cantor set;
    \item a finite union of closed intervals containing a critical point \blue and on which $f$ acts transitively (see Def.~\ref{def:transitive}).
  \end{enumerate}
  %
\end{theoremX}
The last case is the one where the dynamics within the attractor is chaotic,
\blue
in the sense of sensitivity to initial conditions
(see~\cite{Kol04} and~\cite{Rob08} for thorough discussions on alternate definitions of chaotic attractors).
\black
\subsection{Backward asymptotics.}
The asymptotics in the backward direction have been explored much less, even in the two settings mentioned above.
Traditionally, the $\alpha$-limit of a point is defined analogously to the $\omega$-limit:
$$
\alpha_f(x) = \bigcap_{n\geq0}\;\overline{\bigcup_{m\geq n}\{f^{-m}(x)\}}\,.
$$
Notice that, in this case, in general $f^{-m}(x)$ consists in more than a single point and the growth
of the number of counterimages with $m$ can be exponential.
Very few theorems are known on $\alpha_f(x)$. We quote below the major ones known to us:
%
\begin{theoremX}[{\bf $\bm\alpha$-limits in $\bm\bC$}~\cite{Fat20b}]
  \label{thm:aC}
  Let $f:\CPo\to\CPo$ be a rational function of degree larger than 1. Then, for all $z\in\CPo$, with at most two exceptions,
  $\alpha_f(z)\supset J_f$. Moreover, $\alpha_f(z)=J_f$ if and only if $z$ belongs to either $J_f$ or to the basin
  of attraction of some stable cycle of $f$ (but not to the cycle itself).
\end{theoremX}
As in case of $\omega$-limits, also the theorem above has a, much weaker, real analogue:
\begin{theoremX}[{\bf $\bm\alpha$-limits in $\bm\bR$}~\cite{dMvS93}]
  \label{thm:aR}
  Let $f:\RPo\to\RPo$ be a generic analytic map and denote by $Z_f$ its set of critical points.
  Then $\alpha_f(Z_f)=J_f$.
\end{theoremX}
More recently, H.~Cui  and Y.~Ding~\cite{CD10} studied in detail the $\alpha$-limit sets of
S-unimodal maps.
\subsection{$\alpha$-limits are not an optimal analogue of $\omega$-limits.}
While the definition of $\alpha$-limit sets above has the appeal to be the precise analogue, replacing images
with counterimages, of the definition of $\omega$-limit sets, the example below suggests that it does
not represent necessarily the best answer to the question of ``where might $x$ come from''.
\begin{defn}
  A {\em bitrajectory} of a map $f:X\to X$ based at $x\in X$ is a bi-infinite sequence $\{x_i\}_{i\in\bZ}$
  such that:
  \begin{enumerate}
  \item $x_0=x$;
  \item $f(x_i)=x_{i+1}$ for all $i\in\bZ$.
  \end{enumerate}
\end{defn}
\begin{esempio}
  Consider the logistic map $\ell_\mu(x)=\mu x(1-x)$ for some $\mu\in(1,4)$. The point $x=0$ is a repelling
  fixed point and $f(1)=0$, so $\alpha_f(0)=\{0,1\}$. For $x_0$ small enough, the equation
  $\ell_\mu(x)=x_0$ has 2 roots $x_1,\hat x_1$ such that:
  \begin{enumerate}
  \item $x_1\in(0,x_0)$;
  \item $\hat x_1=1-x_1$ has no counterimage under $\ell_\mu$.
  \end{enumerate}
  Hence $f^{-m}(x_0)$ always consists in just two points $x_m,\hat x_m$, with $x_m\to0$ and $\hat x_m\to1$.
  Ultimately, $\alpha_f(x_0)=\{0,1\}$ for all $x_0$ small enough.

  On the other side,
  every bitrajectory $\{t_i\}$ of $\ell_\mu$ passing through $x_0$ satisfies
  $\lim_{n\to-\infty}t_{n}=0$. Indeed, clearly the accumulation points of $\{t_{-i}\}_{i\in\bN}$ belong to $\alpha_f(0)$
  and, as long as $\mu<4$, $\ell_\mu(0.5)<1$ and so there is some neighborhood of 1 which is not in $\ell_\mu([0,1])$.
  In particular, no point close enough to 1 can belong to a bitrajectory of $\ell_\mu$.
\end{esempio}
In fact, an equivalent definition for $\alpha_f(x)$ is given by the following~\cite{Her92}:
\begin{defn}
$\alpha_f(x)$ is the set of $y\in X$
for which there is a sequence of points $\{y_n\}$ and a sequence of monotonically increasing positive numbers $k_n$ such that:
\begin{enumerate}
  \item $f^{k_n}(y_n)=x$;
  \item $\lim_{n\to\infty}y_n=y$.
\end{enumerate}
\end{defn}
Namely, arbitrarily close to each point in $\alpha_f(x)$, there are points whose forward trajectory passes through $x$.

The discussion above shows that, while formally there is a perfect symmetry between the definition of $\alpha$-limit and the one
of $\omega$-limit, when a map is not invertible this symmetry does not extend to the sets obtained from those definitions:
the $\omega$-limit set of a point $x$ contains only (and all) points that are asymptotically approached by $x$ under $f$,
while its $\alpha$-limit contains also points that cannot be asymptotically approached, going backward, by a bitrajectory.

\subsection{$s\alpha$-limits are a better analogue of $\omega$-limits.}
A definition that leads to backward limit sets which are symmetric, from the point of view highlighted above, to $\omega$-limit
sets can be achieved by slightly modifying the definition above:
\begin{defn}[\cite{Her92}]
  Let $x\in X$. The {\em special $\alpha$-limit} $s\alpha_f(x)$ is the set of all $y\in X$ for which there is
  a sequence of points $\{y_n\}$ and a sequence of monotonically increasing positive numbers $k_n$ such that:
\begin{enumerate}
  \item $y_0=x$;
  \item $f^{k_n}(y_n)=y_{n-1}$;
  \item $\lim_{n\to\infty}y_n=y$.
\end{enumerate}
\end{defn}
The symmetry between $\saf$ and $\of$ becomes evident when we write them in terms of limit points of
trajectories $t$.
\begin{defn}
  Let $t$ be a bitrajectory. We denote by $\omega(t)$ (resp. $\alpha(t)$) the set of all forward (resp. backward) limit points of $t$.
\end{defn}
\begin{prop}
  Let $f:X\to X$ and $x\in X$. Then:
  \begin{enumerate}
  \item $\omega_f(x) = \omega(t)$ for any bitrajectory of $f$ based at $x$;
  \item $s\alpha_f(x) = \cup\alpha(t)$, where the union is over all bitrajectories of $f$ based at $x$.
  \end{enumerate}
\end{prop}
%
Not much is known on general properties of $s\alpha$-limit sets. The strongest claim related to them, in
the author's knowledge, is relative to infinite backward trajectories of rational maps on the Riemann sphere:
\begin{theoremX}[\cite{HT03}]
  \label{thm:saC}
  Let $f$ be a rational map of degree larger than 1, $z_0$ a non-exceptional point and $\nu$ the equidistributed
  Bernoullli measure on the space of bitrajectories through $z_0$. Then, for $\nu$-almost all bitrajectories $t$
  passing through $z_0$, we have that $\alpha(t)=J_f$.
\end{theoremX}
Only recently several authors started a
thorough investigation of its general properties, in particular F.~Balibrea, J.~Guirao and M.~Lampart~\cite{BGL13},
J.~Mitchell~\cite{Mit20}, J.~Hant{\'a}kov{\'a} and S.~Roth~\cite{HR20}, S.~Kolyada, M.~Misiurewicz and L.~Snoha~\cite{KMS20}.
In particular, the last three authors formulated the following conjecture:
\begin{conjecture*}[\cite{KMS20}]
  For all continuous maps of the interval, all $s\alpha$-limit sets of points are closed.
\end{conjecture*}
\section{Definitions and graph of S-unimodal maps.}
A {\em discrete dynamical system} on a metric space $(X,d)$  is given by the iterations of a continuous map $f:X\to X$.
In this article we are interested in the case where {\blue $X$ is a closed interval} and $f$ is a S-unimodal map, as defined below:
\begin{defn}[\cite{Sin78}]
  A $C^3$ map $f:[a,b]\to[a,b]$ is {\em S-unimodal} if:
  \begin{enumerate}
  \item either $f(a)=f(b)=a$ or $f(a)=f(b)=b$;
  \item $f$ has a unique critical point $c\in(a,b)$;
  \item the Schwarzian derivative of $f$
    $$
    Sf(x) = \frac{f'''(x)}{f'(x)}-\frac{3}{2}\left(\frac{f''(x)}{f'(x)}\right)^2
    $$
    is negative for all $x\in[a,b]$, $x\neq c$.
  \end{enumerate}
\end{defn}
In all statements and examples throughout the article, we will assume that $c$ is a maximum.
Of course the same proofs hold also when $c$ is a minimum after trivial modifications
that we leave to the reader.

Given any S-unimodal map $f$, the equation $f(x)=y$ has two distinct solutions for each $y<f(c)$.
For each $p\neq c$ in the domain of $f$, we denote by ${\hat p}$ the solution different from $p$
of the equation $f(x)=f(p)$.
\begin{esempio}
  In case of the logistic map $$\ell_\mu(x)=\mu x(1-x)$$ we have that $[a,b]=[0,1]$, $f(0)=f(1)=0$,
  $c=0.5$ and $\hat p = 1-p$. The Schwarzian derivative of $\ell_\mu$ is trivially negative since
  $\ell'''_\mu\equiv0$.
\end{esempio}
\blue
Recall that, as a consequence of Guckenheimer's results in~\cite{Guc79}, each S-unimodal map
is topologically conjugated to a logistic map (see Thm.~6.4 and Remark~1 in~\cite{dMvS93}).
In particular, each S-unimodal map has at most two fixed points.
\black

The following several definitions are needed to define the graph of a dynamical system $f$.
%
\begin{defn}[\cite{Bow75}]
  An {\em $\epsilon$-chain} from $x\in X$ to $y\in X$ is a sequence of points $x_0=x,x_1,\dots,x_n,x_{n+1}=y$
  such that $d(f(x_i),x_{i+1})<\epsilon$ for all $i=0,\dots,n$. We say that $x$ is {\bf downstream} (resp. {\bf upstream}
  from $y$) if, for every $\epsilon>0$, there is an $\epsilon$-chain from $y$ to $x$ (resp. from $x$ to $y$).
  A point $x\in X$ is {\bf chain-recurrent} if it is downstream from itself.
\end{defn}
\subsection{Nodes.} We denote the set of all chain-recurrent points of $X$ under $f$ by $\cR_f$. The relation $x\sim y$ iff $x$ is both upstream
and downstream from $y$ is an equivalent relation in $\cR_f$ (e.g. see~\cite{Nor95b}). The points of $\cR_f/\sim$
will be the nodes of the graph and so we refer to them as {\bf nodes}.
Notice that every node is closed and invariant under $f$~\cite{Nor95b}.
The edges of the graph will be defined through the asymptotics of the system's bitrajectories, as explained below.
%
%
%

%
\begin{prop}[\cite{Nor95b}]
  Given any bitrajectory $t$, there exist nodes $N_1,N_2$ in $\cR_f$, not necessarily distinct, such that $\alpha(t)\subset N_1$ and
  $\omega(t)\subset N_2$.
\end{prop}
\begin{proof}\
  If $x,y\in\omega(t)$ it means that, for every $\epsilon>0$, there are integers $p,q,r$ such that $|x_i-x|<\epsilon$ for $i=p,r$
  and $|x_q-y|<\epsilon$, namely $x$ is both upstream and downstream from $y$ and so they belong to the same node.
  The same argument can be repeated in case of the $\alpha$-limit.
\end{proof}
Notice that, when $\alpha(t)$ and $\omega(t)$ belong to the same node $N$, it means that all points of $t$ actually belong to $N$.
\begin{defn}[\cite{Mil85}]
   A closed invariant set $A$ is an {\bf attractor} if it satisfies the following conditions:
    \begin{enumerate}
        \item the basin of attraction of $A$, namely the set of all $x\in X$ such that $\omega(x)\subset A$, has strictly positive measure;
        \item there is no strictly smaller invariant closed subset $A'\subset A$ whose basin differs from the basin of $A$ by just a zero-measure set.
    \end{enumerate}
\end{defn}
\begin{defn}
  We call a node an {\bf attracting node} if it contains an attractor, otherwise we call it a {\bf repelling node}.
\end{defn}
%
%
%
\blue
\begin{defn}
  We say that a cycle $\gamma$ is {\bf critical}, {\bf regular} or {\bf flip} if, given any $p\in\gamma$, 
  respectively $(f^n)'(p)=0$, $(f^n)'(p)>0$ or $(f^n)'(p)<0$.
\end{defn}
Notice that the definition is consistent because the derivative of $f^n$ is constant on any $n$-cycle.
Recall also that this derivative is zero, and so the cycle is superattracting, if and only if $c$ belongs to the cycle.
\black
\begin{defn}
  We call period-$k$ {\bf trapping region} of $f$ a collection $\cT$ of $k$ closed intervals 
  $J_1,\dots,J_k$ such that:
  \begin{enumerate}
  \item $c$ lies in the interior of $J_1$;
  \item the interiors of the $J_i$ are pairwise disjoint;
  \item $f(J_i)\subset J_{i+1}$, $i=1,\dots,k$, where we use the notation $J_{k+1}=J_1$.
  \end{enumerate}
  We say that $\cT$ is {\bf cyclic} if, moreover,
  \begin{enumerate}
    \setcounter{enumi}{3}
  \item $\partial J_1=\{p_1,\hat p_1\}$ for some periodic point $p_1$;
  \item $f(\partial J_i)\subset\partial J_{i+1}$.
  \end{enumerate}
  We denote by $|\cT|$ the period of $\cT$.
  When $\cT$ is cyclic, we denote by $\Gamma(\cT)$ the periodic orbit passing through $p_1$.
  \blue
  We say that $\Gamma(\cT)$ is the {\bf minimal} cycle of $\cT$ (see Thm.~\ref{thm:RepellingNodes}). In general, we say that a cycle
  $\gamma$ is minimal when $\gamma=\Gamma(\cT)$ for some trapping region $\cT$.
  \black
  Sometimes, to emphasize, we denote by $J_i(\cT)$ the $J_i$ interval of $\cT$.  
  We denote by
  $\Jall=\Jall(\cT)$ the union of the interiors of all the $J_i$.
  We say that a cyclic trapping region is a {\bf flip trapping region} if any two $J_i$ have an endpoint in common,
  otherwise we say it is a {\bf regular trapping region}.
\end{defn}
\blue
\begin{prop}
  A cyclic trapping region $\cT$ is flip (resp. regular) iff $\Gamma(\cT)$ is flip (resp. regular).
  Moreover, $|\cT|=|\Gamma(\cT)|$ if $\cT$ is regular while $|\cT|=2|\Gamma(\cT)|$ if $\cT$ is flip.
\end{prop}
\black
\begin{figure}
  \centering
  \includegraphics[width=15cm]{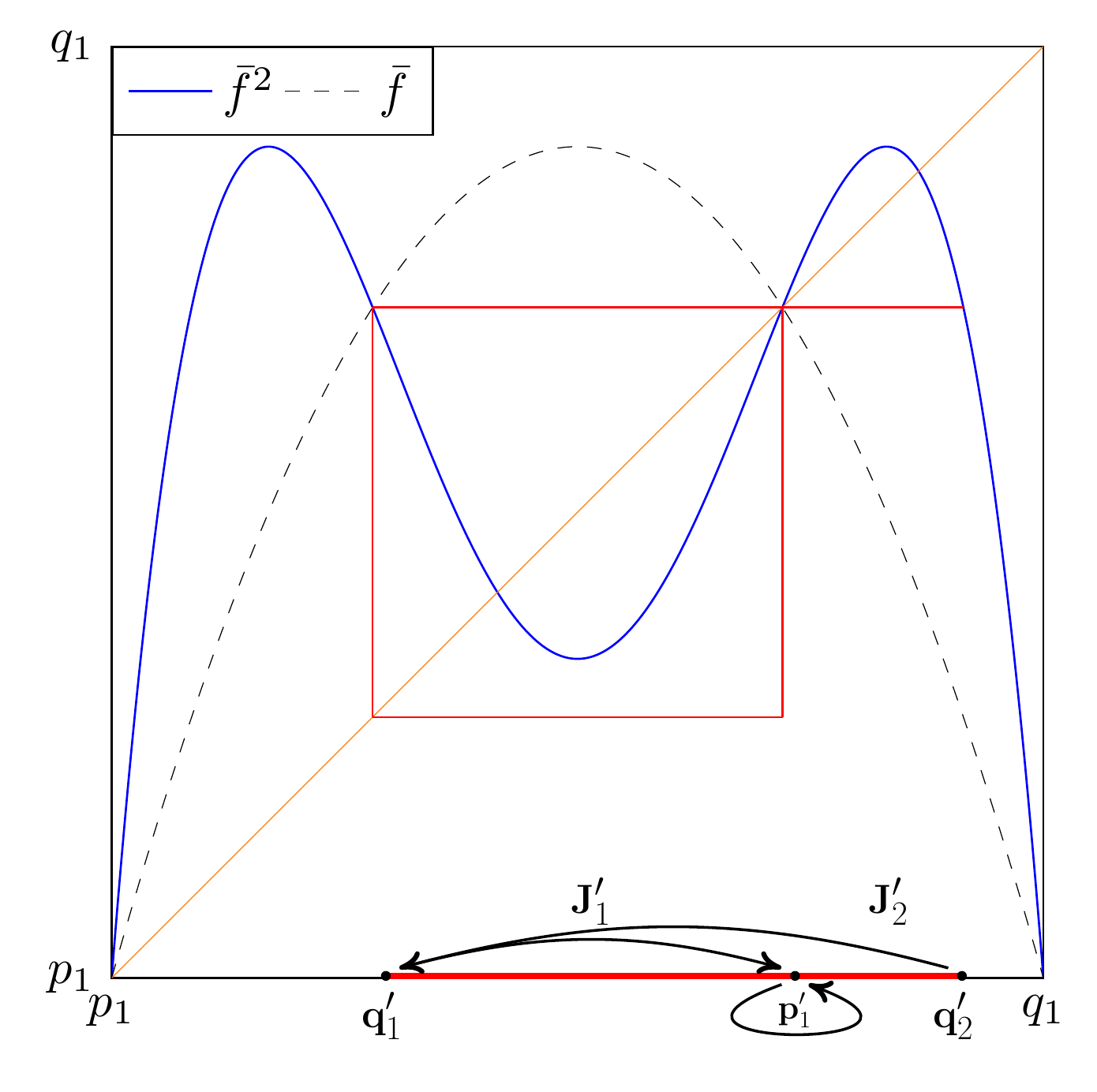}
  \caption{{\bf Example of flip trapping region.} The interval $[p_1,q_1]$, where possibly $q_1<p_1$,
    is the $J_1$ interval of some cyclic trapping region of some period $r$ of a map $f$, so that $J_1$
    is invariant under $\bar f=f^k$. Inside $J_1$ is shown a period-2 flip trapping region $\{J'_1,J'_2\}$
    of $\bar f$. The point $p'_1$, which is fixed by $\bar f$, is an endpoint of both $J'_{1,2}$.
    Both the other endpoints $q'_1,q'_2$ eventually fall on $p'_1$. Notice that $\bar f(c)<q'_2$.
  }
  \label{fig:P1}
\end{figure}
\begin{figure}
  \centering
  \includegraphics[width=15cm]{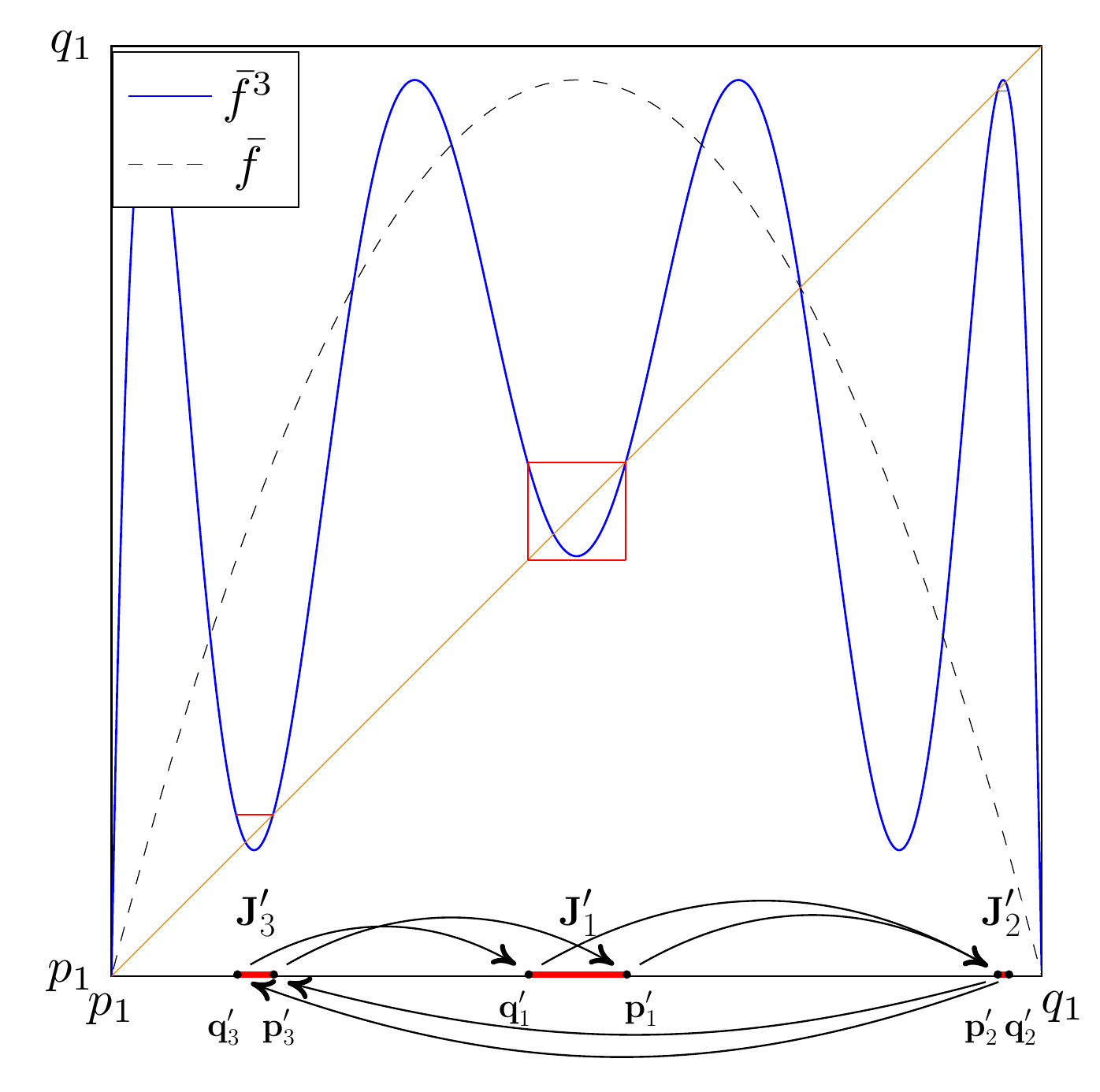}
  \caption{{\bf Example of regular trapping region.} The interval $[p_1,q_1]$, where possibly $q_1<p_1$,
    is the $J_1$ interval of some cyclic trapping region of some period $r$ of a map $f$, so that $J_1$
    is invariant under $\bar f=f^k$. Inside $J_1$ is shown a period-3 regular trapping region $\{J'_1,J'_2,J'_3\}$
    of $\bar f$. The arrows show how the endpoints map under $\bar f$: the endpoints $p'_1,p'_2,p'_3$
    form an unstable 3-cycle of $\bar f$ and all other endpoints $q'_1,q'_2,q'_3$
    eventually fall into this cycle.
  }
  \label{fig:P2}
\end{figure}
%
 %
\begin{esempio}
  The interval $[a,b]$ is, trivially, a period-1 regular cyclic trapping region for every S-unimodal map.
\end{esempio}
\begin{esempio}
  Let $\bar p_\mu$ be the fixed point of the logistic map $\ell_\mu$ other than $0$.
  The collection
  $$
  \cT_\mu=\{J_1=[\ell^2_\mu(c),\bar p_\mu],J_2=[\bar p_\mu,\ell_\mu(c)]\}
  $$
  is a period-2 trapping region for all $\mu\in(1+\sqrt{5},4]$.
  A direct calculation shows that $\cT_\mu$ is {\em cyclic} only at
  $$
    \mu_M=3[1+(19-3\sqrt{33})^{1/3}+(19+3\sqrt{33})^{1/3}]/2\simeq3.67857,
  $$
  which is exactly the point at which $\bar p_\mu$ plunges into a chaotic attractor and ceases to be a node in itself.
  At that point, the logistic map has a {\em crisis}~\cite{GOY82}.
\end{esempio}
This example above can be generalized to S-unimodal maps, as shown below.
\begin{defn}
  \blue
  We call {\bf core} of a S-unimodal map $f$ the interval $[f^2(c),f(c)]$.
\end{defn}
In order to simplify the notation, we often use the notation $c_k$ in place of $f^k(c)$.
\blue
\begin{prop}
  \label{prop:[c2,c1]}
  Let $f$ be a S-unimodal map and $\bar p$ its internal fixed point.
  The core of $f$ is forward-invariant if and only if $\bar p>c$.
\end{prop}
\begin{proof}\
  Notice first of all that, when $c_3<c_2<c$, we have that the restriction of $f$ to $[c_3,c_2]$ is a
  orientation-preserving homeomorphism and so $f([c_3,c_2]) = [c_4,c_3]$. Hence, in general,
  $f([c_k,c_{k-1}]) = [c_{k+1},c_k]$, namely the forward orbit of $c$ is a monotonically decreasing
  sequence contained in $[a,c)$ and therefore it must converge to a fixed point at the left of $c$.
  In particular, when $c<\bar p$, this cannot happen because, when $f$ has two fixed points,
  the boundary one $x=a$ is repelling (recall that $f$, being topologically conjugate to a logistic map,
  cannot have more than two fixed points).
  Now let us go over all possible cases:

  \vspace{2mm}\noindent
  {\bf Case 1: $\bm{\bar p\leq c}$}. Since $f$ is monotonically increasing at the left
  of $c$ and monotonically decreasing at its right, this can happen if and only if the graph of $f$
  is not above the diagonal at $x=c$, namely we must have $c_1\leq c$.
  The case $c_1=c$ is degenerate: in that case $c=\bar p$ is fixed and $[c_2,c_1]$ is a single point.
  When $c_1<c$, the restriction of $f$ to $[c_1,c]$ is a
  orientation-preserving homeomorphism and we fall again in the case outlined above:
  $f([c_1,c])=[c_2,c_1]$ and $f([c_2,c_1])=[c_3,c_2]$, so in particular $[c_2,c_1]$ is not forward
  invariant. 

  \vspace{2mm}\noindent
  {\bf Case 2: $\bm{c<\bar p}$, $\bm{c\leq c_2}$.}  In this case, the restriction of $f$ to $[c_2,c_1]$ is a
  orientation-reversing homeomorphism and so $f([c_2,c_1])=[c_2,c_3]\subset[c_2,c_1]$,
  showing that $c_2<c_3$ and that $[c_2,c_1]$ is forward-invariant. It is not a trapping region,
  though, since $c\not\in(c_2,c_1)$.

   \vspace{2mm}\noindent
  {\bf Case 3: $\bm{c_2<c<\bar p}$.} 
  In this case, we have that
  $$
  f([c_2,c_1])=f([c_2,c]\cup[c,c_1])=[c_3,c_1]\cup[c_2,c_1]=[c_2,c_1]
  $$
  since, by our argument above, the configuration $c_3<c_2<c<\bar p$ cannot take place
  and therefore we must have $c_2\leq c_3$.
  Hence in this case $[c_2,c_1]$ is a period-1 trapping region.
\end{proof}
\black

  Several examples of cyclic flipping and regular trapping regions are shown, respectively, in Figs.~\ref{fig:P1},\ref{fig:T1}
  and~\ref{fig:P2},\ref{fig:T2}.
%
\begin{theoremX}[{\bf Repelling Nodes}~\cite{DLY20}]
  \label{thm:RepellingNodes}
  Let $f$ be a S-unimodal map. Then the minimum distance between a repelling node $N$
  and the critical point of $f$ is achieved at a periodic point $p_1$ of $N$. The closed interval
  with endpoints $p_1$ and $\hat p_1$ is the $J_1$ interval of a cyclic trapping region $\cT(N)$
  such that:
  \begin{enumerate}
  \item $N\cap\Jall(\cT(N))=\emptyset$;
  \item $p_1\in \Gamma(\cT(N))$ and each interval of  $\cT(N)$ has a point of $\Gamma(\cT(N))$ at its boundary;
  \end{enumerate}
  Depending on the type of the corresponding trapping region, there are two types of repelling nodes:
  \begin{enumerate}[label=(\subscript{R}{{\arabic*}})]
  \item A flip cycle. In this case, $\cT(N)$ is a flip trapping region.
  \item A Cantor set with a dense orbit. In this case, $\cT(N)$ is a regular trapping region.
  \end{enumerate}
\end{theoremX}
\begin{figure}
 \centering
 \includegraphics[width=15.5cm]{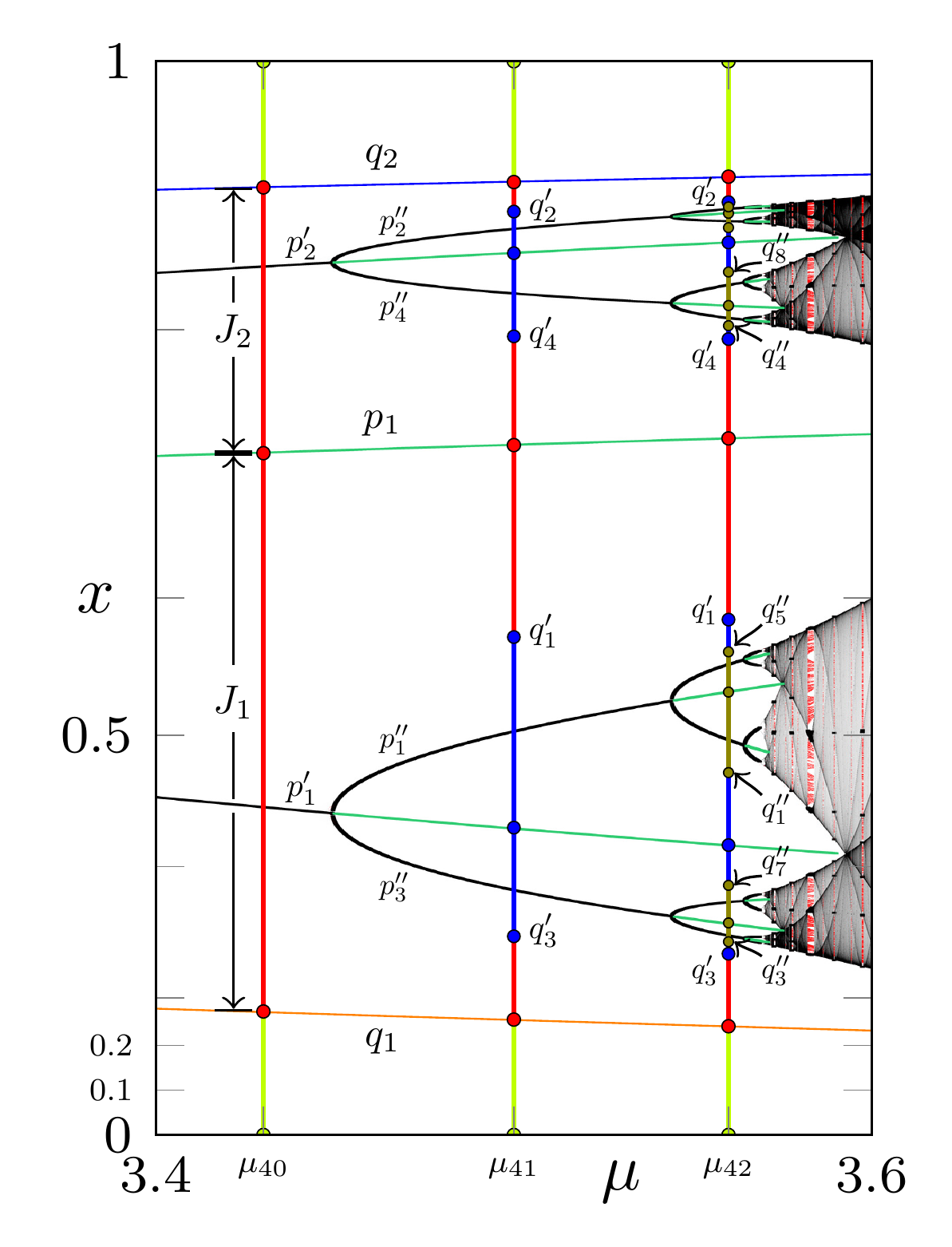}
 \caption{{\bf Examples of flipping trapping regions.} This picture is discussed in detail in Sec.~3.4.
  }
 \label{fig:T1}
\end{figure}
\blue
In order to keep notation light, we often write $\Gamma(N)$ to denote $\Gamma(\cT(N))$ and so on.
\black
\begin{rmk}
  A flip cycle node $N$ is always inside the interior of the union of all $J_i(N)$  (see Fig.~\ref{fig:T1}).
  On the contrary, a Cantor set node $M$ has no point in common with  the interior of the union of all $J_i(M)$
  (see Fig.~\ref{fig:T2}).
\end{rmk}
\begin{theoremX}[{\bf Attracting Nodes}~\cite{DLY20}]
  \label{thm:AttractingNodes}
  Let $f$ be a S-unimodal map. Then its attracting nodes are of the following types:
  
  \begin{enumerate}[label=(\subscript{A}{{\arabic*}})]
  \item An attracting cycle.
  \item An attracting trapping region.
  \item An attracting minimal Cantor set (this is the case when there are infinitely many nodes).
  \item A repelling Cantor set containing a 1-sided attracting periodic orbit.
  \item A trapping region which strictly contains an attracting cyclic trapping region, a repelling Cantor set and part of the basin of attraction.
    The Cantor set and the attractor have a 1-sided attracting periodic orbit in common.
  \end{enumerate}
\end{theoremX}
\begin{rmk}
  Among the attracting and repelling nodes of S-unimodal maps, those of type $A_5$ are the only ones containing
  points that are not non-wandering points.
  In particular, since all points in $\omega(t)$ and $\alpha(t)$ are non-wandering points, this is the only node that cannot be equal 
  to either the $\omega$-limit or the $\alpha$-limit of any bitrajectory.
\end{rmk}
\subsection{Examples of nodes.}\label{sec:bd}
Some of the nodes of the logistic map are shown in Fig.~\ref{fig:bd}.
\begin{figure}
 \centering
 \includegraphics[width=15.5cm]{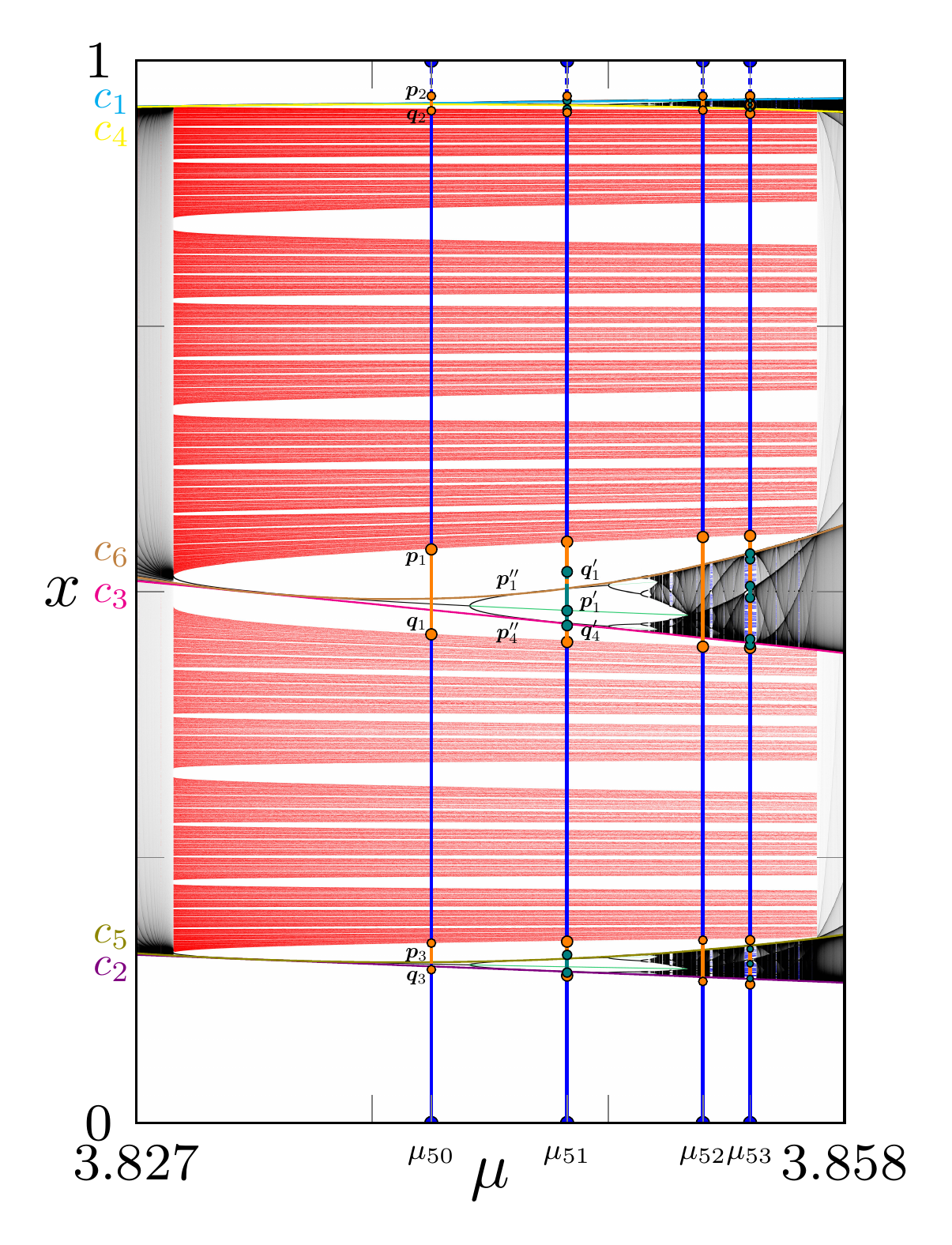}
 \caption{{\bf Examples of regular trapping regions.}  This picture is discussed in detail in Sec.~3.4.
  }
 \label{fig:T2}
\end{figure}

Repelling nodes of type $R_1$ (resp. $R_2$) are painted in green
(resp. red). Each repelling node of either kind exists over some closed connected interval of parameter values and depends
continuously on it. The parameters range over which a repelling Cantor set node $N$ arise is called a {\em window} $W(N)$.
The period of $W(N)$ is the period of the periodic point $p_1(N)$.
The largest window of the logistic map family, shown in detail in Fig.~\ref{fig:T2}, is the period-3 window starting at
$\mu=1+2\sqrt{2}\simeq3.828$ and ending at $\mu\simeq3.857$ (see~\cite{HK03}, p.~299, for an exact expression of this last number).

Attractors are shown in shades of gray. Attractors of type $A_1$ (isolated black lines) and of type $A_2$ are both visible.
Attractors of type $A_4$ and $A_5$ are located, respectively, at the beginning and end of each window~\cite{DLY20}. Attractors of type $A_3$
are not visible because they arise only for a zero-measure set of parameter values~\cite{Lyu02}. The lines visible throughout the chaotic
attractors are iterates of the critical point, about which the density of points under iterations is higher. In colors are highlighted
the first four of these iterates.

\subsection{The Graph.}
Finally we have all ingredients to introduce the graph of a discrete dynamical system:
\begin{defn}[\cite{Con78}]
  We call graph $\Gamma_f$ of the dynamical system $f$ on $X$ the directed graph having as nodes the elements of $\cR_f/\sim$ and having
  an edge from node $N_1$ to node $N_2$ when there exist a bitrajectory $t$ such that $\alpha(t)\subset N_1$ and $\omega(t)\subset N_2$.
\end{defn}
The key result by Conley, obtained originally for continuous systems~\cite{Con78} and generalized later by D.~Norton~\cite{Nor95}
to discrete ones, is that the dynamics outside of nodes is purely of gradient-like nature, as explained below.
\begin{defn}
  A {\bf Lyapunov function}~\cite{WY73} for $f:X\to X$ is a continuous function $L$ such that:
   \begin{enumerate}
       \item $L$ is constant on each node;
       \item $L$ assumes different values on different nodes;
       \item $L(f(x))< L(x)$ if and only if $x$ is not chain-recurrent.
   \end{enumerate}
\end{defn}
\begin{theoremX}[\cite{Nor95}]
    Let $f$ be a dynamical system on a compact metric space $X$. Then there exists a Lyapunov function for $f$. 
    In other words, every point $x\in X$ either belongs to $\cR_f$ or is gradient-like for $f$.
\end{theoremX}
Recall that a loop in a graph is a sequence of nodes $N_1,N_2,\dots,N_k$ such that there is an edge from $N_i$ to $N_{i+1}$
for $i=1,\dots,n-1$ and then an edge from $N_k$ back to $N_1$.
\begin{cor}
  A graph $\Gamma_f$ does not have loops.
\end{cor}
\begin{defn}
  A graph $\Gamma_f$ is a tower if there is an edge between each pair of nodes.
\end{defn}
Notice that, because of the corollary above, if there is an edge from $A$ to $B$ then there cannot be an edge from $B$ to $A$.

In~\cite{DLY20} we prove the following two fundamental results. Notice that the second theorem integrates some original
result of ours together with other classic results by other authors (references are provided within the claim).
\begin{theoremX}[{\bf Tower Theorem.}]
  The graph of a S-unimodal map is a tower.
\end{theoremX}
\begin{theoremX}[{\bf Chain-Recurrent Spectral Theorem}]
  Let $f$ be a S-unimodal map with at least a repellor and denote by $N_0, N_1,\dots,N_p$, where $p\geq1$ is possibly infinite,
  the nodes of $\Gamma_f$ sorted in the order determined by the edges (in particular, $N_p$ is the attracting node of $f$). 
  Then:
    \begin{enumerate}
    \item $N_0$ is the fixed endpoint of $f$.
    \item Each repelling node $N_i$, $0\leq i<p$, determines a cyclic trapping region $\cT(N_i)$ for $f$
      such that $\Gamma(\cT(N_i))\subset N_i$ and each $J_k(\cT(N_i))$ has
      a point of $\Gamma(\cT(N_i))$ as an endpoint. There is no other point of $N_i$ in $\cT(N_i)$~\cite{DLY20}.
      \item The trapping regions $\cT(N_i)$ are nested:
      $$                    
      \overline{\Jall(\cT_{j})}\subset\Jall(\cT_{j-1}),\;1\leq j< p,
      $$
      where, by convention, we take $\cT(N_0)=\{[a,b]\}$~\cite{DLY20}.
    \item Each $N_j$, $1\leq j<p$, is either of type $R_1$ or $R_2$~\cite{DLY20}.
      The measure of each repelling node is zero and the action of $f$ on it is topologically conjugated to
      a subshift of finite type~\cite{vS81}.
    \item Each $N_j$, $0\leq j<p$, is repelling and hyperbolic~\cite{vS81}.
    \item $N_p$ is the unique attracting node of $f$~\cite{Guc79} and it is one of the five types $A_1,\dots,A_5$~\cite{DLY20}.
    \item $N_p$ is hyperbolic unless it is of type $A_4$ or $A_5$~\cite{vS81}.
    \item In each neighborhood of $N_i$, for each $j\geq i$, there are points falling eventually into $N_j$~\cite{DLY20}.
    \item When $p=\infty$, the attracting node $N_\infty$ is a Cantor set of zero Lebesgue measure on which $f$ acts as an adding
      machine~\cite{vS81}.
    \end{enumerate}
\end{theoremX}
\subsection{Examples of trapping regions.}
\label{sec:ex1}
Several cyclic trapping regions are shown in Figs.~\ref{fig:T1} and~\ref{fig:T2} in case of the logistic map. In all cases,
the top node is $N_0=0$ and its corresponding trapping region is the interval $J_1(N_0)=[0,1]$. 

In Fig.~\ref{fig:T1} it is shown the logistic map's bifurcation diagram for $\mu\in[3.4,3.6]$.
Attracting cycles are shown in black. In this range, most of the figure is occupied by the first bifurcation cascade of the diagram.
Throughout this range, the first node $N_0$ is the unstable fixed point 0 and the second node $N_1$ is the other fixed point $p_1$
(painted in green), which is unstable too. We show respectively in orange and blue the points
$q_1=\hat p_1=1-p_1$ and $q_2$, where $q_2$ is the root of $\ell_\mu^2(x)=p_1$ on the other side of $q_1$
with respect to $p_1$ (see Fig.~\ref{fig:P2}). We paint in lime $\cT(N_0)$.


At $\mu_{40}$, the logistic map has a third and final node: 
the attracting 2-cycle $N_2=\{p'_1,p'_2\}$. The flip trapping region $\cT(N_1)$, painted in red,
consists of the following two intervals:
$J_1(N_1)=[q_1,p_1]$ and $J_2(N_1)=[p_1,q_2]$.

Between $\mu_{40}$ and $\mu_{41}$, the 2-cycle $N_2$ bifurcates and becomes repelling and 
a new attracting 4-cycle node $N_3=\{p''_1,\dots,p''_4\}$ arises.
We paint in blue the period-4 flip trapping region $\cT(N_2)$, which consists in the following intervals:
$J_1(N_2)=[p'_1,q'_1]$, $J_2(N_2)=[p'_2,q'_2]$, $J_3(N_2) = [q'_3,p'_1]$ and $J_4(N_2) = [q'_4,p'_2]$.

Similarly, at $\mu_{42}$ node $N_3$ is repelling and we have a new 8-cycle attractor and a new flip
trapping region $\cT(N_3)$ generated by $J_1(N_3)=[q''_1,p''_1]$ consisting in eight intervals,
painted in olive.

Fig.~\ref{fig:T2} shows a detail of the logistic map's period-3 window
(see also Fig.~\ref{fig:U3} for a close-up of its middle cascade).
Besides the nodes, we plot the supplementary lines $c_k$ for $k=1,\dots,6$. 
At the left endpoint of the window, at $\mu_0=1+2\sqrt{2}$, the graph has just two nodes: $N_0$ and an attracting
node $N_1$ of type $A_4$. We paint in blue $\cT(N_0)$.

For all parameter values in the interior of the window, node $N_1$ is the repelling Cantor set
painted in red.
The point $p_1=p_1(N_1)$ belongs to the repelling 3-cycle $p_1,p_2,p_3$ created together with the other
3-cycle $p'_1,p'_2,p'_3$, that is the attracting node for $\mu$ close enough to $\mu_0$ from the right.
The regular trapping region $\cT(N_1)$ of the Cantor node, painted in orange, consists
of the three intervals $J_1(N_1)=[q_1,p_1]$, $J_2(N_1)=[p_2,q_2]$ and $J_3(N_2)=[q_3,p_3]$.
At the right endpoint, the Cantor set has a point in common with the attracting trapping region,
forming an attracting node of type $A_5$.

At $\mu=\mu_{50}$ and $\mu=\mu_{52}$ the graph has just three nodes.
In the first case, the attracting node $N_2$ is the
3-cycle $p'_1,p'_2,p'_3$; in the second case, it is the period-3 trapping region
$[c_4,c_1],[c_2,c_5],[c_3,c_6]$.

At $\mu=\mu_{51}$, the node $N_2$ is repelling and the attracting node $N_3$ is a 6-cycle $\{p''_i\}_{i=1,\dots,6}$.
The flip trapping region $\cT(N_2)$, painted in teal, consists of the six intervals
$J_1(N_2)=[p'_1,q'_1]$, $J_2(N_2)=[p'_2,q'_2]$, $J_3(N_2)=[q'_3,p'_3]$, 
$J_4(N_2)=[q'_4,p'_1]$, $J_5(N_2)=[q'_5,p'_2]$, $J_6(N_2)=[p'_3,q'_6]$.

At $\mu=\mu_{53}$, node $N_2$ is a second Cantor set, painted in blue, namely
this parameter value is within a window inside a window. Point $p_1(N_2)$ has period 9,
so it is a period-9 window within a period-3 window. The attractor $N_3$ in this case
is the period-9 orbit originating from the same bifurcation that created the repelling
period-9 cycle that $p_1(N_2)$ belongs to.
\section{Limit sets of S-unimodal maps.}
\label{sec:LimitSets}
We are now ready to study the $s\alpha$-limit sets of points under $f$.
Throughout the rest of the paper, we will assume that $f$ satisfies the following:

\smallskip
{\bf Assumption (A): $f$ is a S-unimodal map with at least two nodes.}
\smallskip

Under these conditions, the map is not surjective, i.e. $f(c)<b$, and the fixed point $x=a$
is a repelling node, i.e. $f'(a)>1$ \blue (note that $f'(a)>0$ since we are assuming that the critical
point of $f$ is a maximum)\black.
%
\begin{esempio}
  In case of the logistic map $\ell_\mu:[0,1]\to[0,1]$, assumption (A) amounts to $f'(0)>1$ and $f(0.5)<1$.
  These conditions are satisfied if and only if $\mu\in(1,4)$.
\end{esempio}
The nodes of $f$ will be denoted by $N_0,N_1,\dots,N_p$, where $p\geq1$ is possibly infinite,
and sorted in the order determined by the graph's edges (in particular, $N_p$ is the attracting node of $f$).
Recall that $N_0$, under assumption (A), is always equal to the fixed endpoint of $f$.
\blue
\begin{defn}
  For $k=0,\dots,p-1$, we denote by $r_k$ the period of $\cT(N_k)$ and we set
  $$\bar f_i=f^{r_k}|_{J_i(N_k)}:J_i(N_k)\to J_i(N_k),\;i=1,\dots,r_k.$$
  We denote by $K(\bar f_i)$ the core of the map $\bar f_i$ and we call {\bf core of $\mathbf \cT(N_k)$} the collection
  of intervals
  $$
  \cK(N_k) = \{K(\bar f_1),\dots,K(\bar f_{r_k})\}.
  $$
  Finally, we denote by $K(N_k)$ the union of all intervals in $\cK(N_k)$.
\end{defn}
\begin{esempio}
  $\cK(N_0)=\{[c_2,c_1]\}$ since $r_0=1$ is the period of $\cT(N_0)=\{[a,b]\}$.
  If the second fixed point $\bar p$ of $f$ is a repelling node (as in Fig.~\ref{fig:T1}), then $N_1=\{\bar p\}$,
  $r_1=2$ and $\cK(N_1)=\{[c_2,c_4],[c_3,c_1]\}$.
  
  Suppose now that the second node $N_1$ is a Cantor set. Consider for instance the concrete
  case of the logistic map with $\mu=\mu_{50}$ in Fig.~\ref{fig:T2}. Then $\cT(N_1)$ has period
  equal to 3 and $\cK(N_1)=\{[c_2,c_5],[c_3,c_6],[c_4,c_1]\}$.
\end{esempio}
%
\begin{prop}
  The following properties hold for each $k=0,1,\dots,p-1$:
  %
  \begin{enumerate}
  \item $K(\bar f_{r_k-i}) = [c_{2r_k-i-1},c_{r_k-i-1}]$;
  \item $K(N_k)\subset\Jall(N_k)$;
  \item for $k\leq p-2$, $\cK(N_k)$ is a trapping region and $\Jall(N_{k+1})\subset K(N_k)$.
  \end{enumerate}
\end{prop}
\begin{proof}\
  \blue
  Recall that $J_1(N_k)=[p_1(N_k),\hat p_1(N_k)]$ is the interval of $\cT(N_k)$ containing $c$, so that
  $$K(\bar f_1) = [\bar f_1^2(c),\bar f_1(c)] = [c_{2r_k},c_{r_k}].$$
  Moreover, each map $f|_{J_i(N_k)}:J_i(N_k)\to J_{i+1}(N_k)$, $i=2,\dots,r_k$, where we set $J_{r_k+1}(N_k)=J_1(N_k)$,
  is a homeomorphism for $i>1$, and so $K(\bar f_{r_k}) = f^{-1}(K(\bar f_1))\cap J_{r_k}(N_k)=[c_{2r_k-1},c_{r_k-1}]$ and, ultimately,
  $$
  K(\bar f_{r_k-i}) = [c_{2r_k-i-1},c_{r_k-i-1}].
  $$
  The assumption that $N_k$ is not the last node implies the following two facts:
\begin{enumerate}
  \item 
  $K(\bar f_{1})\subset int(J_1(N_k))$. Otherwise, we would have $\bar f_1(c)=p_1(N_k)$ or $\bar f_1(c)=\hat p_1(N_k)$
  and, in both cases, $N_k$ would be the attractor and there could be no other node besides it. Hence, in general,
  $K(\bar f_{i})\subset int(J_{i}(N_k))$ for $i=1,\dots,r_k$, which proves point (2).
  \item
    Since the cycle $\Gamma(N_k)$ is repelling and, for $k<p-1$,  $N_{k+1}$ is
  repelling too, then by Proposition~\ref{prop:[c2,c1]} we have that the core of each $\bar f_i$
  is forward-invariant with respect to $\bar f_i$. Indeed, each $\bar f_i$ is a S-unimodal map on $J_i(N_k)$ whose
  non-boundary fixed point $\bar p_i$ must be at the right of the critical point, since otherwise
  $\bar p_i$ would be attractive and there would be no other node inside $\Jall(N_k)$.
  Hence
  $$
  f(K(\bar f_1)) = f^2(K(\bar f_{r_k})) = \dots = f^{r_k}(K(\bar f_2)) \subset  K(\bar f_2)
  $$
  and, more generally, $f(K(\bar f_i))\subset f(K(\bar f_{i+1}))$, $i=1,\dots,r_k$, where we put $\bar f_{r_k+1}=\bar f_1$.
\end{enumerate}

  Since the attractor is unique, it must lie inside $K(N_k)$ and so $\Jall(N_{k+1})$ must be completely
  contained inside $K(N_k)$ as well. In particular, $c\in J_1(N_{k+1})\subset K(\bar f_1)$,
  so $\cK(N_k)$ is a trapping region, proving point (3).
\end{proof}
\begin{defn}
  We define sets $U_1,\dots,U_p$ as follows:
  \begin{itemize}
  \item $U_p=N_p$;
    \item $U_{k}=K(N_{k-1})\setminus K(N_{k})$ for $1\leq k< p-1$;
    \item if $p<\infty$, $U_{p-1}=K(N_{p-2})\setminus N_p$; 
  \end{itemize}
  \noindent
  Finally, we set $U_0=[a,c_2)$ and $U_{-1}=(c_1,b]$. 
  We say, for short, that $x\in[a,b]$ is a {\em level-$k$ point} if $x\in U_k$.
\end{defn}
\begin{prop}
  The sets $U_{-1},\dots,U_p$ satisfy the following properties:
  \begin{enumerate}
  \item $[a,b]=\bigsqcup_{i=-1}^pU_i$;
  \item $N_k\subset U_k$ for $k=0,\dots,p$;
  \item $U_k\subset\Jall(N_{k-1})$ for $k=1,\dots,p$.
  \end{enumerate}
\end{prop}
\begin{proof}\
  1. By construction, the $U_i$ are all pairwise disjoint. Assume first that $p<\infty$. Then
  $$
  \bigcup_{i=1}^pU_i = (K(N_0)\setminus K(N_1)) \cup (K(N_1)\setminus K(N_2)) \cup \dots
  \cup (K(N_{p-2})\setminus N_p)\cup N_p = K(N_0) = [c_2,c_1]
  $$
  so that $\cup_{i=-1}^pU_i=[a,c_2)\cup[c_2,c_1]\cup(c_1,b]=[a,b]$. If $p=\infty$, then
  $$
  N_p = \bigcap_{k\geq1} \overline{\Jall(N_k)} = \bigcap_{k\geq1} K(N_k) 
  $$
  and so, even in this case, $\cup_{k\geq1}U_k\cup U_\infty=[a,b]$.
   
  2,3. Since $U_k\subset K(N_{k-1})\subset \Jall(N_{k-1})$, $U_k\cap N_i=\emptyset$ for $i\leq k-1$.
  Since $U_{k+1}\subset K(N_{k})\subset \Jall(N_{k})$, $U_{k+1}\cap N_k=\emptyset$. More generally, since
  $K(N_{k+1})\subset K(N_k)$, $U_{i}\cap N_k=\emptyset$ for $i\geq k+1$. Since every point in $[a,b]$
  belongs to some $U_i$, the only possibility if that $N_k\subset U_k$ for each $k=0,\dots,p-1$.

\end{proof}
\black
\begin{figure}
 \centering
 \includegraphics[width=16.5cm]{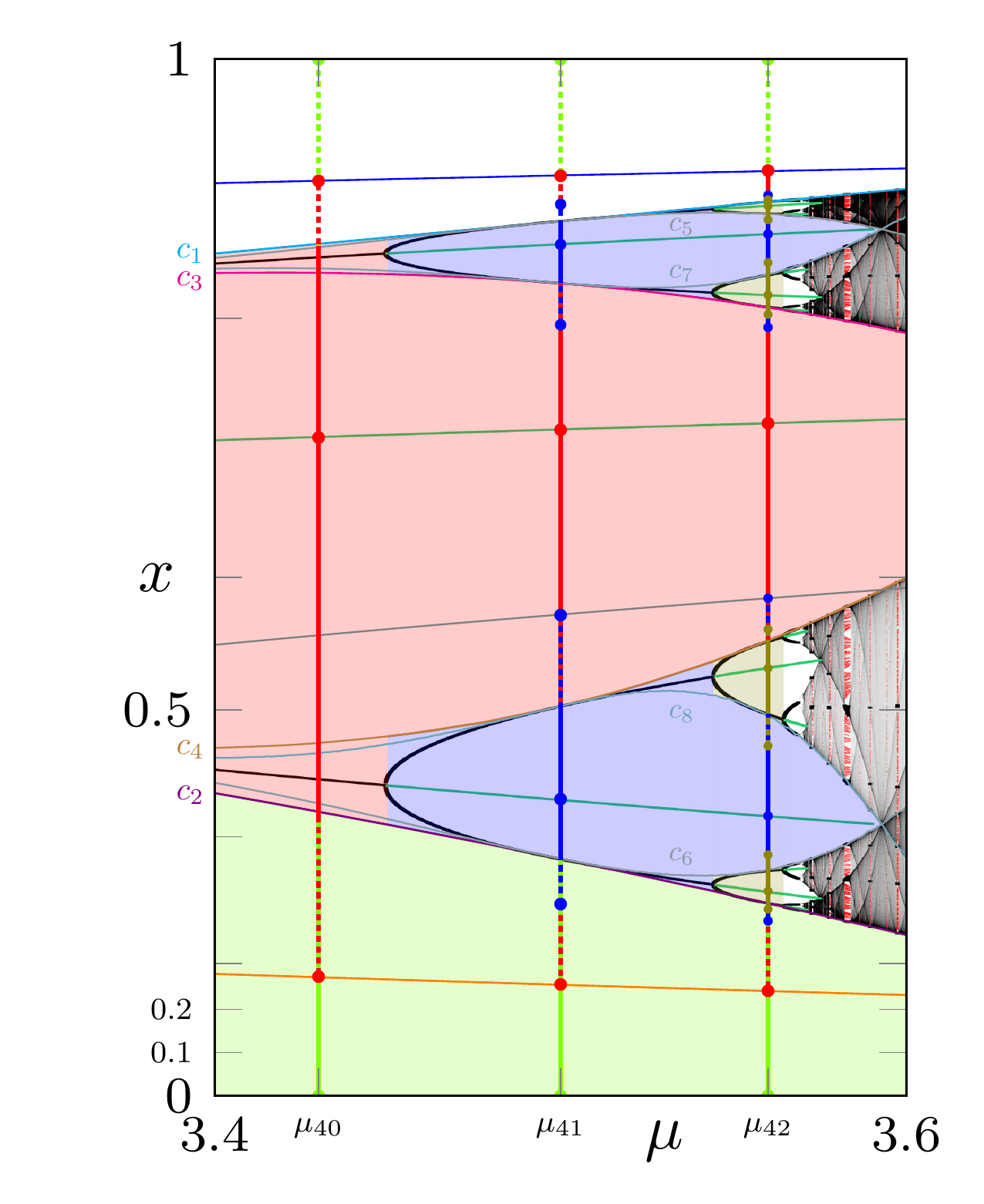}
 \caption{{\bf Examples of $U_k$ sets in case of flip trapping regions.} This picture is discussed
   in detail in Sec.~\ref{sec:ex2}.
  }
 \label{fig:U1}
\end{figure}
\begin{esempio}
  Consider the case, shown in Fig.~\ref{fig:U1}, of the logistic map $f$ with $\mu=\mu_{40}$.
  This map has three nodes: $N_1$ is the fixed point other than $x=0$ and the attractor $N_2$
  is a 2-cycle $\{p'_1,p'_2\}$.
  In this case
  $$
  U_2=\{p'_1,p'_2\},\;
  U_1=[c_2,p'_1)\cup(p'_1,p'_2)\cup(p'_2,c_1].
  $$
\end{esempio}
\begin{figure}
 \centering
 \includegraphics[width=16cm]{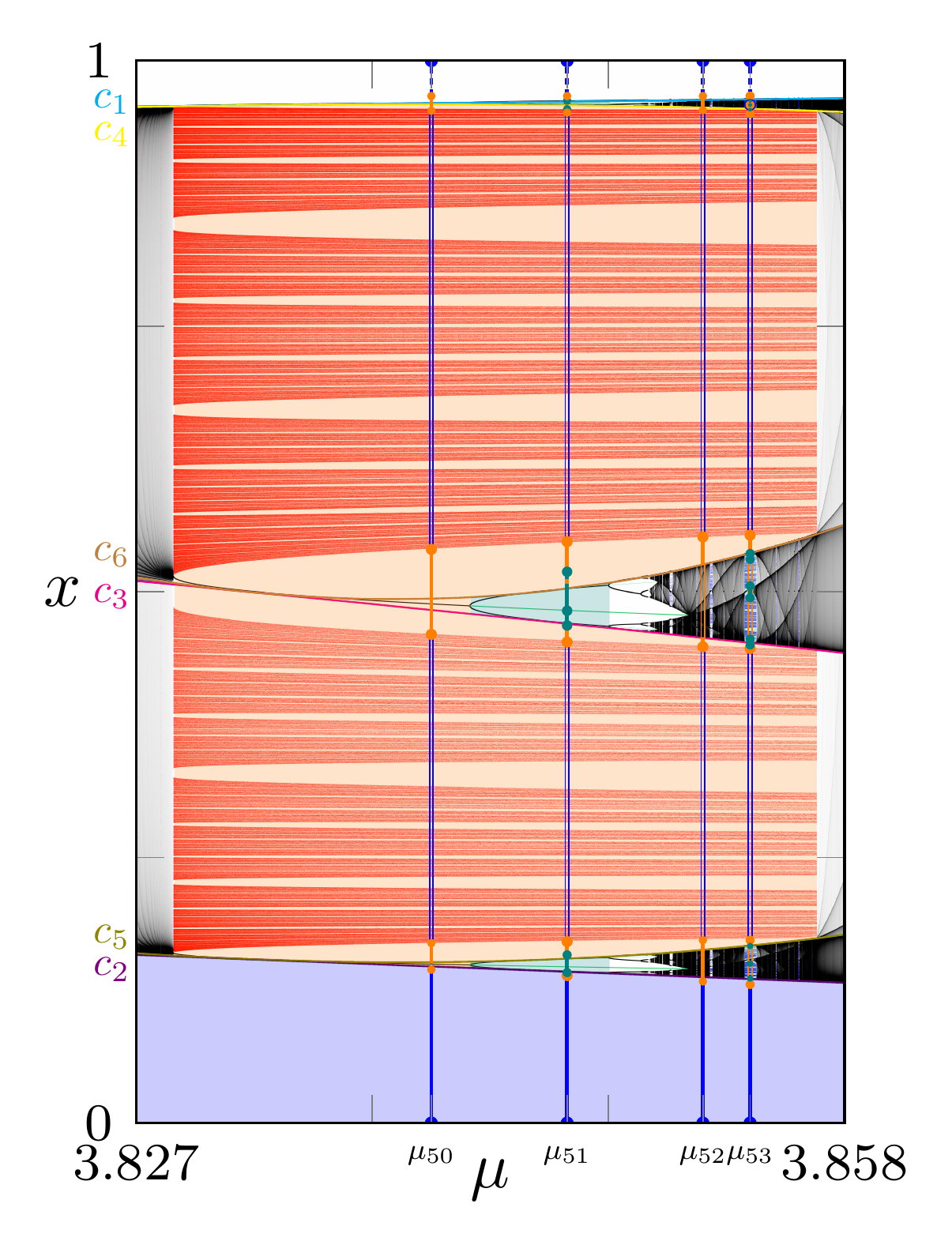}
 \caption{{\bf Examples of $U_k$ sets in case of regular trapping regions.} This picture is discussed
   in detail in Sec.~\ref{sec:ex2}.
  }
 \label{fig:U2}
\end{figure}
\begin{figure}[ht]
 \centering
 \includegraphics[width=15.5cm]{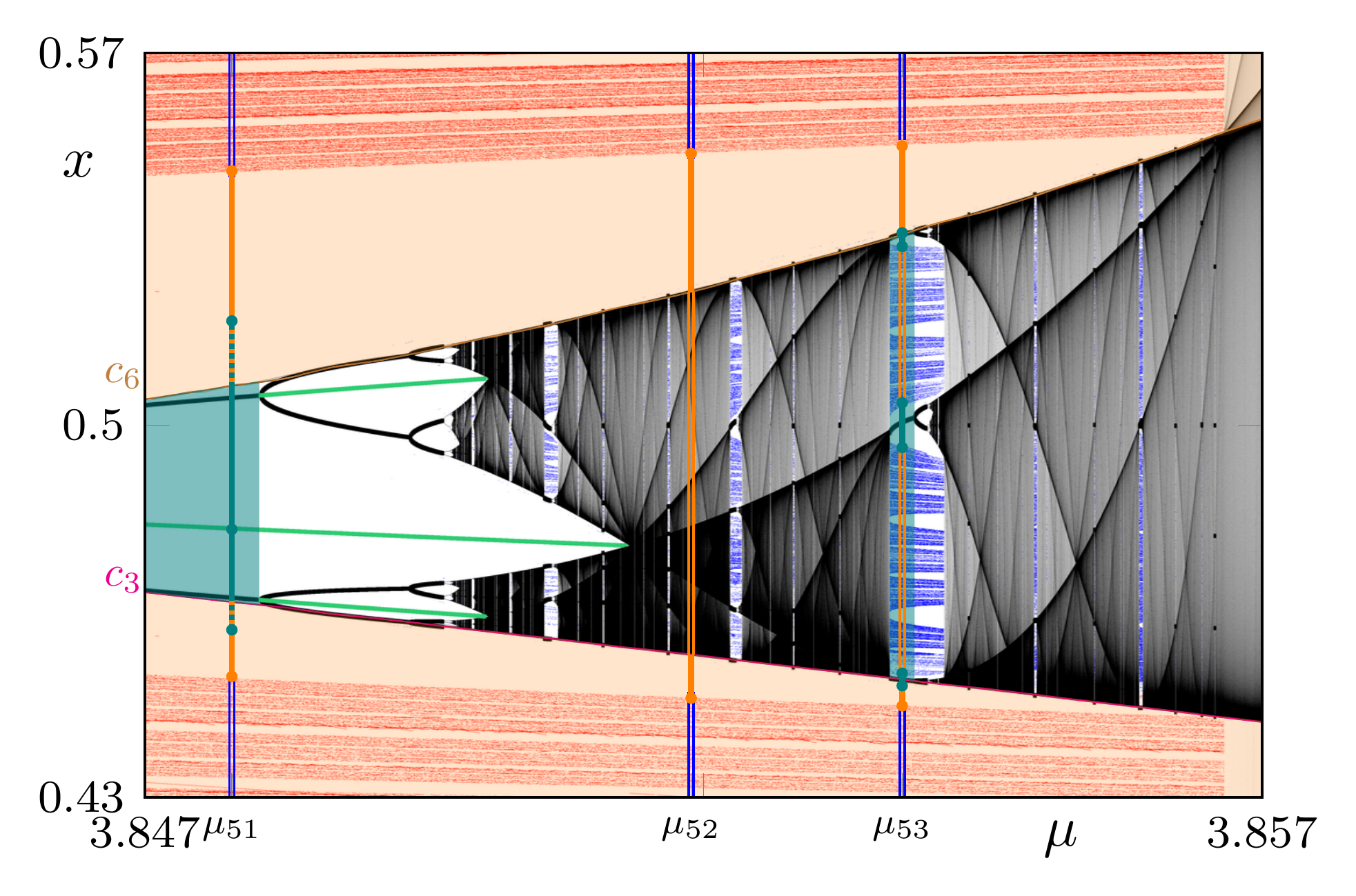}
  \caption{{\bf Close-up of the central cascade in Fig.~\ref{fig:U2}.} This picture is discussed
   in detail in Sec.~\ref{sec:ex2}.
  }
  \label{fig:U3}
\end{figure}
\begin{esempio}
  Consider the case, shown in Fig.~\ref{fig:U2}, of the logistic map with $\mu=\mu_{53}$.
  This map has four nodes: $N_1$ is a Cantor set, painted in red; $N_2$ is a second Cantor set,
  painted in blue; the attractor $N_3$, three points of which are visible (in black) in Fig.~\ref{fig:U3},
  is a 9-cycle.
  In this case,
  $$
  U_3=N_3,\;
  U_2=[c_2,c_5]\cup[c_3,c_6]\cup[c_4,c_1]\setminus N_3,\;
  U_1=(c_5,c_3)\cup(c_6,c_5).
  $$
\end{esempio}
%
%
%
%

\subsection{Forward limits.}
\label{sec:f}
While the focus of this article is on backward limits, for completeness we include also the following result about the relation between
the forward limit of a point and its position.
\begin{thm}[{\bf $\bm\omega$-limits of S-unimodal maps.}]
  Assume that $f$ satisfies (A) and let $x\in U_k$. Then $\omega_f(x)\subset N_i$ for some unique $i\geq k$.
\end{thm}
\begin{proof}\
By the Chain-Recurrent Spectral Theorem, we know the following:
  \begin{enumerate}
    \item Trapping regions are nested one into the other and are all forward invariant;
    \item $N_j\subset\Jall(N_k)$ for all $j>k$;
    \item $N_j\cap\Jall(N_k)=\emptyset$ for all $j\leq k$;
    \item $N_k\cap\overline{\Jall(N_k)}$ is a single periodic orbit.
  \end{enumerate}
  %
  Hence, the $\omega$-limit set of any $x$ which
  lies into $U_k$ must be contained in one of the nodes $N_i$ for some $i=k,\dots,p$.
\end{proof}
In particular, this means that the set of possible outcomes for $\omega_f(x)$ is a locally constant monotonically
decreasing function of the distance of $x$ from the critical point $c$.
\begin{rmk}
  By Theorem~\ref{thm:oR}, $\omega(x)$ must be either a cycle or a Cantor set or an interval cycle.
  Moreover, for Lebesgue-almost all $x$, $\omega(x)=N_p$~\cite{Guc79,Guc87}. Since all other nodes are repelling,
  \blue
  when $\omega(x)\subset N_i$ for some $i\neq p$ the only possibility is that
  \black
  there is a forward trajectory starting at $x$ that actually falls on $N_i$ in a finite number of iterations.
  Equivalently, such $x$ must belong to a backward trajectory from some point in $N_i$.
\end{rmk}
%
We turn now to the main goal of this article, which is studying the possible $s\alpha$-limits
of a point depending on its level, namely depending on which $U_k$ set it belongs to.
Following Conley's observation that points either belong to some node or are gradient-like, we
discuss separately these two cases. All these partial results will be finally summarized in
Thm.~\ref{thm:main} (Sec.~\ref{sec:main}).
%
\subsection{Backward limits within repelling nodes.}
\label{sec:br}
As we recalled earlier, repelling nodes are either cycles (type $R_1$) or
\blue
transitive (see Def.~\ref{def:transitive}) Cantor sets (type $R_2$).
\black
The restriction of $f$ to either case is topologically conjugated to a subshift of finite type, a fact that was
proved first by van Strien in 1981~\cite{vS81} using kneading theory.
For sake of clarity and self-completeness, we sketch below a more explicit proof of this fact that
will also lead us to establish that $s\alpha_f(x)\supset N$ for each point $x$ belonging to a repelling node $N$.

{\bf The case of node $N_1$ in the period-3 window of $\bm{\ell_\mu}$.}
Consider the concrete case of the period-3 window $W$ of the logistic map $f=\ell_\mu$ (see Figs.~\ref{fig:bd} and ~\ref{fig:U2}).
For every $\mu$ in the interior of $W$, the point $p_1(N_1)$ belongs to a repelling 3-cycle and
$\cT(N_1)=\{J_1,J_2,J_3\}$ is a regular period-3 trapping region. Since the $J_i$ are disjoint,
the set $C=[0,1]\setminus\cup_{k=0}^\infty f^{-k}(J_1)$ of all points that do not ever fall in $J_1$ is the
complement of a dense countably infinite union of disjoint intervals and so is a Cantor set (see Prop.~3.21
in~\cite{DLY20} for more details).
 
{\bf Decomposing $\bm C$.}
Recall that, {\blue when $\bar p>c$}, each node different from $N_0$ lies inside $[f^2(c),f(c)]$ and that, since $c$ belongs to either the
basin of the attractor or to the attractor itself, each repelling node actually lies in $L_c=(f^2(c),f(c))$.
Hence $N_1\subset C\cap L_c$.
Furthermore, by construction we have that $$N_1\subset C\cap\left(L_c\setminus\Jall(N_1)\right).$$
The set $L_c\setminus\Jall(N_1)$ is the disjoint union of the two closed intervals
$$A_0=[p_3,1-p_1]\hbox{ and } A_1=[p_1,p_2]$$
and so, correspondingly, we can write
$$C\cap L_c=C_0\sqcup C_1,$$ with $C_i=C\cap A_i$, $i=0,1$.

{\bf A map $\bm{s:C_0\sqcup C_1\to\Sigma_2}$.}
Exploiting the decomposition above, one can associate to the trajectory $\{f^n(p)\}$ of each point $p\in C\cap L_c$
an element $s(p)=(s_0,s_1,\dots)$ of the single-sided shift space $\Sigma_2$ in two symbols, say 0 and 1, so that $s_n=0$
if $f^n(p)\in C_0$ and $s_n=1$ otherwise.

For instance, for every $\mu\in W$, the map $f$ has a fixed point $\bar p$ inside
$C_1$, so
$$s(\bar p)=(111\dots).$$
\blue 
Since $f^2$ is a polynomial of degree 4 and both fixed points of $f$ are also fixed points
of $f^2$, for every $\mu\in W$ the map $f$ has a single 2-cycle
$\{\bar q_1,\bar q_2\}$, $\bar q_1<\bar q_2$. Note that it is impossible that $\bar q_2<c$ since $f(\bar q_1)=\bar q_2$
would imply $f(\bar q_2)>\bar q_2$.
Moreover, since within $W$ this 2-cycle is repelling, then necessarily $f^2(c)<c$ and so $c\in(\bar q_1,\bar q_2)$.
\black
Hence,
$$
 s(\bar q_1)=(0101\dots),\;\;
 s(f(\bar q_1))=s(\bar q_2)=(1010\dots).
$$
 Finally, in case of the 3-cycle $\Gamma(N_1)$, we have that
 $$
 s(p_1(N_1)) = (110110\dots), \;
 s(f(p_1(N_1)))=s(p_2)=(101101\dots), \;
 s(f^2(p_1(N_1)))=s(p_3)=(011011\dots).
 $$

Notice that, by construction,
$$s(f(p)) = \sigma(s(p)),$$
where $\sigma:\Sigma_2\to\Sigma_2$ is the standard shift operator $\sigma((a_1a_2a_3\dots)) = (a_2a_3\dots)$. 

{\bf Each $\bm{s(p)}$ misses the word ``00''...}
While $(111\dots)$ is in the range of $s$, $(000\dots)$ is not, since the other fixed point of $f$,
namely the node $N_0$, does not belong to $N_1$. This is an example of the following more general property:
$\sigma(p)$ cannot contain the word ``00''.
\blue
Indeed, since $[p_3,1-p_1]\subset[0,c]$,
$$
f(A_0)=f([p_3,1-p_1]) = [f(p_3),f(1-p_1)]=[p_1,p_2]=A_1
$$
while, since $[p_1,p_2]\subset[c,1]$,
$$
f(A_1)=f([p_1,p_2]) = [f(p_2),f(p_1)]=[p_3,p_2]=A_0\cup J_1(N_1)\cup A_1.
$$
\black
Since $C\cap L_c$ is invariant under $f$, this means ultimately that
$$
f(C_0)\subset C_1,\;
f(C_1)\subset C_0\sqcup C_1.
$$
In particular, if the $n$-th component of $s(p)$ is 0, then its $(n+1)$-th component must be 1.

{\bf ...but misses no other word.}
Denote by $W_{00}$ the subset of $\Sigma_2$ of all elements not containing the word ``00''.
Clearly $W_{00}$ is a subshift of finite type, namely it is invariant under $\sigma$ and is defined by a finite
number of conditions (in fact, by just one condition). We claim that, for each element $w=W_{00}$, there is $p\in C\cap L_c$
such that $s(p)=w$.

\blue
Indeed, let $w=(w_1w_2w_3\dots)$ and define the following sets:
\begin{eqnarray*}
I_{w,1} &=& A_{w_1}\\
I_{w,2} &=& A_{w_1}\cap f^{-1}(A_{w_2})\\
I_{w,3} &=& A_{w_1}\cap f^{-1}(A_{w_2})\cap f^{-2}(A_{w_3}) = A_{w_1}\cap f^{-1}(A_{w_2}\cap f^{-1}(A_{w_3}))\\
I_{w,4} &=& A_{w_1}\cap f^{-1}(A_{w_2})\cap f^{-2}(A_{w_3}) \cap f^{-3}(A_{w_4}) = A_{w_1}\cap f^{-1}(A_{w_2}\cap f^{-1}(A_{w_3}\cap f^{-1}(A_{w_4})))\\
\dots
\end{eqnarray*}
By construction, $I_{w,k}$ is the open subset of $C$ containing all points whose first $k$ symbols
coincide with the first $k$ symbols of $w$.
Set $I_w=\cap_{k=1}^\infty I_{w,k}$. Then $I_w\subset C_0\sqcup C_1$ since the orbit of any point of $I_w$
never leaves $A_0\sqcup A_1$ and $s(p)=w$ for every $p\in I_w$.

In order to show that $I_w$ is non-empty for each $w\in W_{00}$, we point out that,
since $f$ is unimodal, the counterimage of any interval not containing $c$ is the union of
two intervals lying on opposite side with respect to $c$. Given the action of $f$ on $A_0$ and $A_1$, we furthermore
have that the counterimage under $f|_{A_0\sqcup A_1}$ of any interval in $A_0$ is a subinterval of $A_1$ while the
counterimage of any interval in $A_1$ is a pair of intervals, one in $A_0$ and one in $A_1$.
Hence, from the expression of $I_{w,k}$, it is clear that this set is empty if and only if, for some $r\leq k$, the interval
$A_{w_r}\cap f^{-1}(\dots)$ lies in $A_0$ (so $w_r=0$) and, at the following
step $A_{w_{r-1}}\cap f^{-1}(A_{w_r}\cap f^{-1}(\dots))$, we have $A_{w_{r-1}}=A_0$ (namely $w_{r-1}=0$).
Ultimately, this argument shows that $I_w$ is empty if and only if $w$ contains the word ``00''.
\black

{\bf $\bm s$ is injective}. Since $I_w\subset C_0\sqcup C_1$ is a closed interval and $C_0\sqcup C_1$
is a Cantor set, the only possibility is that $I_w$ is a single point, namely $s:C_0\sqcup C_1\to W_{00}$
is a bijection.


{\bf $\bm s$ is a homeomorphism}. A basic neighborhood $U_w$ of a point $w_0\in W_{00}$ is the set of all
$w\in W_{00}$ that coincide with $w_0$ up to their $r$-th component for some $r>0$. The set $s^{-1}(U_w)$
is the set we denoted above by $I_{w,r}$ and, as we pointed out above, this is always a neighborhood of
$s^{-1}(w)$. Similarly, due to the continuity of $f$, given any $p\in C_0\sqcup C_1$, for each $M>0$
there is a neighborhood $V_p$ of $p$ such that $s(p)$ and $s(p')$ coincide up to their first $M$
symbols for each $p'\in V_p$

{\bf $\bm{f|_{N_1}}$ is topologically conjugated to $\bm \sigma$}. This is an immediate consequence of the
fact that, as pointed out above, $s\circ f=\sigma\circ s$ for all points of $N_1$ and that $s$ is a homeomorphism.

{\bf $\bm{N_1=C\cap L_c}$}. Because of the points above, it is enough to show that there is an $\epsilon$-chain
between every two points of $s(C\cap L_c)$. Given $x,y\in s(C\cap L_c)$ and $\epsilon>0$, there is an element $z$
whose distance from $x$ is smaller than $\epsilon$ so that all components $z_i$ up to some finite index $n=n(\epsilon)>0$
coincide with the components $x_i$ while $z_{n+i}=y_i$ for all $i>0$. Then the sequence
$$
x,z,\sigma(z),\dots,\sigma^n(z)=y
$$
is an $\epsilon$-chain between $x$ and $y$.

Hence, $N_1$ is the set of all points in $[f^2(c),f(c)]$ that do not fall into $\Jall(N_1)$ under $f$.

{\bf The general case.}
\blue 
Consider the general case of a Cantor set retractor $N_k$. Denote by $r$ the period of the trapping region
$\cT(N_{k-1})$ and set $\bar f_i=f^r|_{J_i(N_{k-1})}$. Then each $\bar f_i:J_i(N_{k-1})\to J_i(N_{k-1})$
is a S-unimodal map having first node $\overline N_{i,0}=\{p_i(N_{k-1})\}$ and second node the repelling
Cantor set $\overline N_{i,1}=N_k\cap J_i(N_{k-1})$. Note that
$
|\cT(\overline N_{i,1})|=|\cT(N_{k})|/|\cT(N_{k-1})|,\;i=1,\dots,r.
$

Similarly to the case above, $\overline N_{i,1}\subset[c_{2r-i-1},c_{r-i-1}]$ and the set
$(c_{2r-i-1},c_{r-i-1})\setminus\Jall(\overline N_{i,1})$ is the disjoint union of $r-2$ closed intervals
$A_{i,1},\dots,A_{i,r-2}$. Hence, like we did above, we can associate to each point 
$p\in (c_{2r-i-1},c_{r-i-1})\setminus\cup_{k=0}^\infty \bar f_i^{-k}(J_1(\overline N_{i,1}))$ an element $\sigma_i(p)$
belonging to the single-sided shift space $\Sigma_{r-2}$. Correspondingly to the action of $\bar f_i$ on the
$A_{i,k}$, there will be some number of words that will not appear in any $\sigma_i(p)$. Since there
is a finite number of $A_{i,k}$, there will be a finite number of forbidden words.  Similarly to how we did above, one can prove that these maps
$\sigma_i$ are continuous and injective and all elements of $\Sigma_{r-2}$ without those forbidden words in their sequence
are the image of some $p$. Hence, the image of each $\sigma_i$ is a subshift of finite type.

Ultimately, we illustrated the following result:
\black
\begin{theoremX}[\cite{vS81}]
  Let $N$ be a repelling node of a S-unimodal map $f$. Then $f|_N$ is topologically conjugate
  to a subshift of finite type.
\end{theoremX}

\subsection{Backward dense orbits in subshifts of finite type.}
Shifts have been thoroughly studied because they appear ubiquitously in every dynamics subfield. 
Nevertheless the following elementary result on backward dynamics, in the author's knowledge, did
not appear explicitly so far in literature.
\begin{thm}
  Let $S$ be a 1-sided subshift of finite type with a dense trajectory.
  Then each point has a backward dense bitrajectory.
\end{thm}
\begin{proof}\
  Let $F$ be the full 2-sided shift in the same symbols $\alpha_1,\dots\alpha_k$ of $S$. 
  A bitrajectory passing through an element $x=(x_1x_2x_3\dots)\in S$ is an element
  $X=(\dots x_{-3}x_{-2}x_{-1}\,.\,x_1x_2x_3\dots)\in F$.
  
  Recall that, in a (1- or 2-sided) shift space $X$ with a dense trajectory, for any two words
  $w,w'\in X$, there is a word $u\in X$ such that $wuw'\in X$~\cite{Boy00}.
  Since in $S$ there is an element $d=(d_1d_2d_3\dots)$ with a dense trajectory, the
  same is true about $F$: given any element $Y=(\dots y_{-3}y_{-2}y_{-1}\,.\,y_1y_2y_3\dots)$
  there is some finite word $u_1\dots u_r$ such that
  $$
  D=(\dots y_{-3}y_{-2}y_{-1}u_1\dots u_r\,.\,d_1d_2d_3\dots)\in F.
  $$
  By Sec.~9.1 in~\cite{LM95}, every irreducible 2-sided shift (namely a 2-sided shift with a
  dense trajectory) has a doubly transitive point $T=(\dots t_{-3}t_{-2}t_{-1}\,.\,t_1t_2t_3\dots)$,
  namely a point whose orbit is dense both backward and forward.

  Now, given any $x=(x_1x_2x_3\dots)\in S$, there is a finite word $u_1\dots u_s$
  such that $X=(\dots t_{-3}t_{-2}t_{-1}u_1\dots u_s\,.\,x_1x_2x_3\dots)\in F$. This
  point represents precisely a bitrajectory passing through $x$ whose backward 
  orbit is dense by construction. 
\end{proof}
As a corollary, we obtain our main result on retractors.
\begin{thm}
  Every point of a repelling node $N$ of $f$ has a bitrajectory backward dense in $N$.
\end{thm}
%
\subsection{Backward limits within attracting nodes.}
Of the five types of attracting nodes, only types $A_2$ and $A_5$, namely the two cases when the attractor
is chaotic, are non-trivial. Indeed, cases $A_1$ (stable periodic orbit) and $A_4$ (single-sided stable periodic
orbit belonging to a Cantor repellor) are covered by the results of the previous subsection. About case $A_3$,
we recall that it is well-known that, in a minimal set, every point has a dense backward orbit~\cite{KS09}.

{\bf Backward orbits in a chaotic attractor.}
\blue
Let $A$ be a chaotic attractor, namely an attracting trapping region. 
Recall that in $A$, as well as in every retracting node, there is a dense set of periodic points (but 
$c$ does not belong to any of them, since any cycle containing $c$ is superattracting).
In each retractor node $N_k$ we have a special cycle, the minimal cycle $\Gamma(N_k)$, which is the cycle
passing through $p_1(N_k)$, the closest point of $N_k$ to $c$. On the contrary, none of the cycles within
$A$ can be minimal, since $A$ belongs to the attracting node and nodes are disjoint.
Nevertheless, in analogy with minimal cycles, we give the following standard definition.
%
\begin{defn}
  Given $p,q\in[a,b]$, we say that $p$ is {\bf closer to $\bm c$} than $q$ if $(p,\hat p)\subset(q,\hat q)$, where as usual
  $\hat p$ is the root of $f(x)=f(p)$ other than $p$.
  Given a cycle $\gamma$, we denote by $p_1(\gamma)$ the point of $\gamma$ that is the closest
  to $c$ and we set $J_1(\gamma)=[p_1(\gamma),\hat p_1(\gamma)]$.
  We say that the cycle $\gamma$ is closer to $c$ than the cycle $\gamma'$ when
  $p_1(\gamma)$ is closer to $c$ than $p_1(\gamma')$.
\end{defn}
\begin{rmk}
  The relation among cycles given by ``being closer to $c$'' is a linear order in the set of all cycles of $f$.
  In case of the logistic map, the distance between $p_1(\gamma)$ and $c$ is the same as between
  $\hat p_1(\gamma)=1-p_1(\gamma)$ and $c$ and so the ``being closer'' relation in this case
  is literal.
\end{rmk}
%

Minimal cycles are quite special: either there are only
finitely many of them (when $N_p$ is not of type $A_3$) or $p_1(\Gamma(N_k))\to c$ (otherwise).
On the contrary, the set of accumulation points of the set of all points $p_1(\gamma)$ of all cycles $\gamma$
of a S-unimodal map $f$ having at least one repelling Cantor set is large (easy to build examples
in subshifts of finite type).
%
Next definition highlights a type of cycle which represents a slight generalization minimal ones and is of central
importance in the chaotic case.
%
\begin{defn}[\cite{Guc79}]
  Given a regular (resp. flip) $n$-cycle $\gamma\subset A$, we say that $p_1(\gamma)$ is
  {\bf regular central} (resp. {\bf flip central}) if $f^n$ (resp. $f^{2n}$) is monotonic
  on $(p_1(\gamma),c)$, namely if $c\not\in f^i(J_1(\gamma))$ for $i=1,\dots,n-1$
  (resp. $i=1,\dots,2n-1$).
When $p_1(\gamma)$ is central, we also say that $\gamma$ itself is central.
Finally, we say that a closed interval $H$ is a {\bf homterval} for $f$ if $c\not\in f^k(H)$ for any integer $k\geq0$,
namely if all iterates $f^k$, $k=0,1,\dots$, are strictly monotonic on $H$.
\end{defn}
%
%
\begin{esempio}
  Every minimal cycle $\Gamma(\cT)$ is central and its point $p_1(\cT)$ is a central point.
\end{esempio}
\begin{esempio}
  Consider the unique 2-cycle $\gamma_2=\{q_-,q_+\}$ of the logistic map. 
  For each $\mu\in[1+\sqrt{5},4]$, $\gamma_2$ is flip. For $\mu\in[1+\sqrt{5},\mu_{MF}]$, where $\mu_{MF}$
  is the Myrberg-Feigenbaum point, $\gamma_2$ is a node. Hence, close enough from the right to $\mu_{MF}$,
  $\gamma_2$ is a flip central point. On the other side, at $\mu=4$ we have that 
  $q_\pm=\frac{5\pm\sqrt{5}}{8}$, so that $q_-<c<q_+$ and $q_-$ is the closest to $c$. Set $J_1=[q_-,1-q_-]$.
  Hence
  $$
  f(J_1) = [q_+,1]
  ,\;
  f^2(J_1) = f([q_+,1]) = [0,q_-]
  ,\;
  f^3(J_1) = f([0,q_-]) = [0,q_+]
  $$
  Since $c\in f^3(J_1)$, $\gamma_2$ is not central. For $\mu$ close enough to 4 from the left,
  $c_1$ is close to 1 and both $c_2$ and $c_3$ are close to 0, so that still $c\in f^3(J_1)$.
\end{esempio}
\begin{prop}
  \label{prop:pbar}
  Let $f$ be a S-unimodal map and $\bar p$ its internal fixed point.
  There is no flip cycle if $\bar p\leq c$.
  For $c<\bar p$, the lowest-period flip central point of $f$ is $\bar p$.
  If $f$ has a regular 3-cycle $\gamma$, then $p_1(\gamma)$ is
  the lowest-period regular central point of $f$.
\end{prop}
\begin{proof}\
  Set $\bar J_1=[\hat{\bar p},\bar p]$. Then $f(J_1)=[\bar p,c_1]$,
  which does not contain $c$, and so $\bar p$ is flip central.
  Consider now the case when $f$ has a regular 3-cycle $\{p_1,p_2,p_3\}$.
  Set $J_1=[\hat p_1,p_1]$. Then
  $f(J_1) = [p_2,c_1]$ and $f^2(J_1)=[c_2,p_3]$. Clearly $c$ does not belong
  to either one of these two sets and so $p_1$ is central.
\end{proof}
%
%
%
%
The following facts are crucial for our results below:
\begin{theoremX}[\cite{Guc79}]
  \label{thm:Guc}
  Let $f$ have a chaotic attractor $A=N_p$ and set $r=|\cT(N_{p-1})|$.
  Then:
  \begin{enumerate}
  \item $c\in A$;
  \item $\alpha_f(c)=[a,b]$;
  \item for any neighborhood $U$ of the critical point, there is a central cycle $\gamma$ with $J_1(\gamma)\subset U$;
  \item for each central $s$-cycle $\gamma\subset A$, with $s>r$ if $\gamma$ is regular and $2s>r$ if it is flip,
    there is a central cycle $\gamma'\subset A$, closer to $c$ than $\gamma$, and a $k>0$ such that $f^k(J_1(\gamma'))\supset J_1(\gamma)$;
  \item $f$ has no homtervals.
  \end{enumerate}
\end{theoremX}
The following proposition is implicitly used in a proof in~\cite{Guc79} (see also Thm.~II.7.9 in~\cite{CE80}).
\begin{prop}
  \label{prop:central}
  The set $C_f$ of all central points of a S-unimodal map $f$ is finite if and only if $N_p$ is
  of type $A_1$ or $A_4$. When $C_f$ is infinite, the only point of accumulation of the set
  of all central points of $f$ is the critical point. 
\end{prop}
\begin{proof}\
  Suppose first that there are infinitely many central points. Consider a sequence $q_k$ 
  of central points of period $n_k$ such that $q_k$ is closer to $c$ than $q_{k-1}$
  for all $k$ and suppose that $q_k\to q_\infty\neq c$. Then there is some closed non-trivial interval
  $H\subset (q_\infty,1-q_\infty)$ not containing the critical point.
  Now, notice that $n_k\to\infty$ because there is a finite number of cycles of any given period.
  Hence $c\not\in f^n(H)$ for any $n\geq0$ since, for any given $n$, there are $k$ for which $n_k>n$
  and for such $k$ we must have that $c\not\in f^i(q_k,1-q_k)$ for $i=1,\dots,n_k-1$.
  So $H$ is a homterval, which contradicts point (5) of the theorem above.

  By point (3) of the previous theorem, when the attractor is chaotic (namely $N_p$ is of type
  $A_2$ or $A_5$) there are infinitely many central points. When $N_p$ is of type $A_3$,
  there are infinitely many nodes $N_k$ and each point $p_1(\cT(N_k))$ is central.
  When the attractor is a cycle $\gamma$ (namely $N_p$ is of type $A_1$ or $A_4$), then
  $\gamma$ is minimal and the interval $J^1(\gamma)$ is contained in the basin of
  attraction of $\gamma$ and so cannot contain any point belonging to a cycle.
\end{proof}
%


%
{\bf The case $\bm{N_1=A}$}. In this case there are only two nodes, the fixed endpoint $N_0$ and the chaotic
attractor $N_1=A$, which is of type $A_2$.
The chaotic attractor must consist into a single interval, since the non-internal fixed point $\bar p$ of $f$ must belong
to it (or there would be a third node) and a trapping region with more than one component cannot contain
any fixed point. Since $c\in A$, then $A=[c_2,c_1]$.

\blue
{\bf The general case}. In case $p>1$, set $r=|\cT(N_{p-1})|$. Then the $J_i(N_{p-1})$ are invariant
under $f^r$ and we set $\bar f_i = f^r|_{J_i(N_{p-1})}$. Then each $\bar f_i$ has just two nodes:
the repelling fixed point $p_i(N_{p-1})$ and the chaotic attractor $N_p\cap J_i(N_{p-1})=[c_{2r-i-1},c_{r-i-1}]$.
Hence, we reduce again the the previous case.


In order to prove the existence of backward dense orbits, we use the two general results below about
topological dynamical systems.
\begin{defn}[\cite{Bir20}]
  \label{def:transitive}
  A dynamical system $f:X\to X$  is {\bf transitive} if it has a dense orbit. 
\end{defn}
Several properties equivalent to topological transitivity can be found in~\cite{KS97}.
%
\begin{theoremX}[\cite{Mal12}]
  \label{thm:tr}
  Each transitive map $f$ has a dense set of points with a backward dense orbit.
\end{theoremX}
\begin{defn}[\cite{Par66}]
  \label{def:vstransitive}
  A dynamical system $f:X\to X$  
  is {\bf very strongly transitive} if, for every open set $U\subset X$, there is a $n\geq1$ such that $\cup_{i=1}^n f^i(U)=X$.
\end{defn}
Recently Akin, Auslander and Nagar gave the following characterization of very strongly transitive systems.
\begin{theoremX}[\cite{AAN16}]
  \label{thm:vstr}
  A dynamical system $f:X\to X$ is very strongly transitive if and only if for any $x\in X$ and $\epsilon>0$
  there is a $N>0$ such that $A_N(x)=\bigcup_{i=1}^N f^{-i}(x)$ is $\epsilon$-dense, namely $A_N(x)$
  intersect all balls of radius $\epsilon$ in $X$.  
\end{theoremX}
\begin{prop}
  The restriction of $f$ to $N_p=A$ is very strongly transitive.
\end{prop}
\begin{proof}\
  The general case argument above shows that it is enough to consider the case $p=1$.
  Since $c$ has a dense set of counterimages in $A$ and there are central cycles arbitrarily close to $c$,
  we can assume without loss of generality that $U=int(J_1(\gamma))$ for some central cycle $\gamma$.

  As shown in Prop.~\ref{prop:pbar}, the fixed point $\bar p$ is a flip central point for $f$.
  Set $\bar J_1=[\hat{\bar p},\bar p]$. Then we have that
  $f(\bar J_1)=[\bar p,c_1]$ and $f^2(\bar J_1) = [c_2,\bar p]$, so that
  $$A=f(\bar J_1)\cup f^2(\bar J_1),\;\; f(\bar J_1)\cap f^2(\bar J_1)=\{\bar p\}.$$
  
  
  Starting from the central cycle $\gamma_0=\{\bar p\}$, we can build a sequence of central cycles $\gamma_k$
  so that, for each $k\geq1$, $J_1(\gamma_k)\subset J_1(\gamma_{k-1})$ and there is a $n_k\geq2$ such that
  $f^{n_k}(J_1(\gamma_k))\supset J_1(\gamma_{k-1})$. In particular, for each $k\geq1$ there is a $N_k$ such that
  $f^{N_k}(J_1(\gamma_k))\supset J_1(\gamma_0)$.
  Since, by Proposition~\ref{prop:central}, central points can only accumulate at $c$, then $\lim_{k\to\infty}p_1(\gamma_k)=c$
  and so there is some $r\geq1$
  such that $\gamma_r$ is closer to $c$ than $\gamma$ and so
  $f^{n_r}(J_1(\gamma))\supset J_1(\gamma_0).$
  Hence $f^{n_r+1}(J_1(\gamma))\supset [\bar p,c_1]$ and $f^{n_r+2}(J_1(\gamma))\supset [c_2,\bar p]$, namely
  $
  \bigcup_{i=1}^{n_r+2}f^i(J_1(\gamma))= A.
  $
\end{proof}
%
%
Now we are ready to prove our main result on chaotic attractors.
\begin{thm}
  Every point of a chaotic attractor $A$ of $f$ has a bitrajectory backward dense in $A$. In particular, every point of
  an attracting node $N$ of type $A_2$ has a backward orbit dense in $N$.
\end{thm}
\begin{proof}\
  By Thm.~\ref{thm:tr} applied to the restriction of $f$ to $A$, there are points in $A$ having a backward orbit
  dense in $A$.
  Let $b=(b_1,b_2,...)$ be one of these backward dense orbits and let $x_0$ be any point of $A$.
  Let $x_1$ be a counterimage of $x_0$ of some order $n_1\geq0$
  closer to $b_1$ than $1/1$. Such counterimage exists by Thm.~\ref{thm:vstr}.
  Then let $x_2$ be a counterimage of $x_1$ of some order $n_2\geq 0$ closer to $b_2$ than $1/2$ and so on. All these
  counterimages exist by Thm.~\ref{thm:vstr}. The backward orbit $(x_0,x_1,x_2,\dots)$ is easily seen to be dense as well.
\end{proof}
%

\black
{\bf The case of nodes of type $A_5$}. As we pointed out in Remark~2, a node of type $A_5$
contains, besides a chaotic attractor which is also a cyclic trapping region, the following
two sets:
1) a repelling Cantor set, with which it shares a cycle, and 2) intervals of points
that are not non-wandering. Since no backward trajectory can asymptote to points that
are not non-wandering, in this case no point has a bitrajectory which is backward
dense in the whole node.
\subsection{Backward limits of non-chain-recurrent points.}
As pointed out initially by Conley in the continuous finite-dimensional case, and
extended later by Norton to the discrete case~\cite{Nor95}, all bitrajectories through
non-chain-recurrent points have their backward and forward limits belonging to different
nodes, so that they determine an edge of the graph.
Below we present two examples that turn out to be typical.

{\bf Backward limits of level-0,1 points when $N_1$ is a repelling cycle.}
This happens, for instance, in case of the logistic map $f=\ell_\mu$ for $\mu\in(3,\mu_M)$.
We already discussed in Example~5 the structure of the sets $U_k$.

Points in $U_{-1}$ have no counterimage, so $s\alpha_f(x)=\emptyset$ if $x\in U_{-1}$.

Each $x\in U_0$ has two counterimages but one of them lies in $U_{-1}$, so there is
a single bitrajectory passing through each such $x$. Since $x=0$ is backward attracting,
this bitrajectory asymptotes backward to 0. Hence, $s\alpha_f(x)=N_0$ if $x\in U_{0}$.

More generally, each point in $[0,c_1]$ has a backward trajectory asymptoting to $0$
since any interval $[0,\epsilon]$, $\epsilon>0$, will eventually cover the whole image
of $f$ under iterations, so $s\alpha_f(x)\supset N_0$ if $x\in[0,c_1]$.
Indeed, recall that an interval not containing $c$ such that none of its iterates contain
$c$ is called a {\em homterval} (see~\cite{DLY20} for more details on homtervals and on the
argument below). As a consequence of a fundamental result of Guckenheimer~\cite{Guc79},
the only configuration when a homterval can arise in a S-unimodal map is when
the attractor $N_p$ is a cycle and the node $N_{p-1}$ is also a cycle. In that case,
the interval between a point $x\in N_p$ and a point $y\in N_{p-1}$ such that no other node
point lies in $[x,y]$ is a homterval. No other type of homterval can arise. 
This means that, in all possible cases, under iterations $[0,\epsilon]$ will eventually
cover $c$ and so, in the following step, will cover all points in the image of the map.

Now consider the period-2 cyclic trapping region of $N_1$, namely the two intervals
$$
J_1(N_1)=[q_1,p_1], J_2(N_1)=[p_1,q_2]
$$
where $q_1=\hat p_1$ and $f(q_2)=q_1$. Notice that both $J_1$ and $J_2$ are forward
invariant under $\bar f=f^2$ and let us set $g=f|_{J_2}$.
Since $c\not\in J_2$, the map $g:J_2\to J_1$ is a diffeomorphism.
For the same reasons above, within $J_1$ only points in $[\bar f(c),p_1]=[c_2,p_1]$
have a backward trajectory asymptoting to $p_1$. Hence, in $J_2$, only points in
$[g^{-1}(c_2),g^{-1}(p_1)]=[p_1,c_1]$ have a backward trajectory asymptoting to $p_1$.
This means ultimately that the only points with backward trajectories asymptoting
to $p_1$ are the points in the interval $[c_2,c_1]$. Since $U_1\subset[c_2,c_1]$,
we have that $s\alpha_f(x)=N_0\cup N_1$ if $x\in U_1$.


{\bf Backward limits of level-0,1 points when $N_1$ is a repelling Cantor set.}
This happens, for instance, in case of the logistic map $f=\ell_\mu$ for $\mu\in(1+2\sqrt{2},3.868\dots)$.
We already discussed in Example~6 the structure of the sets $U_k$.

Just like the case above, points in $U_{-1}$ have no bitrajectories passing through them
and all points in $[0,c_1]$ have a backward trajectory asymptoting to $0$.

As discussed in Sec.~4.2, the Cantor set $N_1$ is the set of all chain-recurrent points
in $[c_2,c_1]$ that do not lie in $\Jall(N_1)$. Each point $x\in [c_2,c_1]$ is either a point
of $N_1$ or lies in a counterimage of $J_1(N_1)$. The case of points of $N_1$ has been
discussed in Sec.~4.2.
\blue
In the other case, since by definition $N_1=[c_2,c_1]\setminus\cup_{k\geq0}f^{-k}(J_1(N_1))$,
every point in $[c_2,c_1]\setminus N_1$ belongs to a backward trajectory starting in the interior of $J_1(N_1)$
and asymptoting to $N_1$.
\black
In any case, then, $s\alpha_f(x)=N_0\cup N_1$ if $x\in U_1$.
\subsection{Main theorem.}
\label{sec:main}
We are now ready to prove the
main result of the article.
\blue
\begin{thm}[{\bf $\bm{s\alpha}$-limits of S-unimodal maps}]
  \label{thm:main}
  Assume that $f$ satisfy (A). If $x$ is a level-$k$ point, $k<p$, then $s\alpha_f(x)=\bigcup\limits_{i=0}^kN_i$.
  If $k=p$ and $N_p$ is not of type $A_5$, then $s\alpha_f(x)=\bigcup\limits_{i=0}^pN_i=\Omega_f$ for all level-$p$
  points. If $N_p$ is of type $A_5$, denote by $C$ and $A$ respectively the repelling Cantor set 
  and the attractor contained in $N_p$.
  Then $s\alpha_f(x)=\Omega_f$ if $x\in A$ and $s\alpha_f(x)=\bigcup\limits_{i=0}^{p-1}N_i\cup C=\Omega_f\setminus int(A)$
  otherwise.
\end{thm}
\begin{proof}\
  We start by pointing out that, for $k=0,\dots,p-1$, $K(N_k)$ is the set of all points of $\overline{\Jall(N_{k})}$
  that have a backward trajectory asymptoting to $N_{k+1}$.
  Indeed, let $r$ be the period of $\cT(N_k)$ and set $\bar f = f^{r}$. Then each $J_i(N_k)$ is forward-invariant
  under $\bar f$ and after restricting $\bar f$ to each of the $J_i(N_k)$ we reduce to either one of the two case examples
  in Section~4.5.

  Since $K(N_{k})\subset K(N_{k'})$ for $k'\leq k$, then every point of $K(N_{k})$ has also a
  backward trajectory asymptoting to $N_{k'+1}$ for $k'\leq k$. Since $U_k\subset K(N_{k-1})$, then
  $s\alpha_f(x)\supset\cup_{i=0}^kN_i$ for every $x\in U_k$. Moreover, since
  $U_k\cap\Jall(N_{k+1})=\emptyset$, then no point of $U_k$ can asymptote to nodes with $k'>k$.
  This shows that $s\alpha_f(x)=\cup_{i=0}^kN_i$ for every level-$k$ point, $k<p$.

  Consider now the case $k=p$. As showed in Sec.~4.4, each point of the attracting node $N_p$ has a backward
  trajectory dense in $N_p$ for all types of attracting node except $A_5$, so in all those cases
  $s\alpha_f(x)=\cup_{i=0}^pN_i=\Omega_f$ for every level-$p$ point.

  When $N_p$ is of type $A_5$, the node contains points which are not non-wandering and these points
  cannot be obtained as limit of backward trajectories.
  Recall that, in this case, the attractor is a cyclic
  trapping region and $C$ is the set obtained from $N_p$ after removing from it all counterimages
  of the component of $A$ containing the critical point $c$.
  Hence, if $x\in A$ then its $s\alpha$-limit contains both $A$, because each
  point in a chaotic attractor has a backward dense orbit, and $C$, because of the relation
  mentioned above between $A$ and $C$, so that $s\alpha_f(x)=\Omega_f$.
  For the same reasons, if $x\in N_p\setminus A$, then $s\alpha_f(x)=\cup_{i=0}^{p-1}N_i\cup C$.
\end{proof}
%
\black
\subsection{Examples of $U_k$ sets.}
\label{sec:ex2}
Several examples of $U_k$ sets are shown in Figs.~\ref{fig:U1}-\ref{fig:U3} in the concrete case of the
logistic map $\ell_\mu$.

Fig.~\ref{fig:U1} uses Fig.~\ref{fig:T1} as background, although we omit the points labels to improve
the picture's readability. Over this background, we overlap the curves $c_k=\ell_\mu^k(c)$, $1\leq k\leq8$,
and we shade points of the sets $U_0$ in lime and $U_1$ in red throughout the whole range of $\mu$;
we paint the sets $U_2$ in blue in the range when node $N_2$ is repelling and is equal to the 2-cycle $\{p'_1,p'_2\}$;
finally, we paint in olive the sets $U_3$ in the range when node $N_3$ is repelling and is equal to the
4-cycle $\{p''_1,\dots,p''_4\}$ for $\mu$ close enough to the bifurcation of the 2-cycle attractor.

Note that we shade points in $U_k$ with the same color as $\cT(N_k)$. We leave $U_{-1}$, the set of points
through which passes no bitrajectory, in white.
Moreover, we dash in two colors the parts of a trapping region $\cT(N_k)$ that are contained in a
$U_{k'}$ with $k'<k$. For instance, for each point $x$ of $\cT(N_2)$ that is dashed in blue and red,
we have $s\alpha(x)=N_0\cup N_1$, while for each point $y$ of $\cT(N_2)$ that is painted in full blue
and is not on the attractor we have that $s\alpha(x)=N_0\cup N_1\cup N_2$.

Close to the left boundary of the picture, the logistic map has three nodes: $N_0$ is the boundary
fixed point, $N_1$ the internal fixed point and $N_2$ a 2-cycle attractor. For all $\mu\in(3.4,3.6)$,
all points of $[0,c_1]$ have a bitrajectory asymptoting backward to $N_0$ and all points in
$K(N_0)=[c_2,c_1]$ have a bitrajectory
asymptoting backward to $N_1$. The only points with bitrajectories asymptoting backward to the attractor
$N_2$ are the points of the attractor itself.

At $\mu_0=1+\sqrt{6}\in(\mu_{40},\mu_{41})$, the 2-cycle undergoes a bifurcation, becomes repelling and a new
attracting 4-cycle $N_3$ arises. As soon as $N_2$ becomes repelling, all points in $K(N_1)=[c_2,c_4]\cup[c_3,c_1]$
(and no other point) have a bitrajectory asymptoting backward to it. This type of behavior repeats
at each bifurcation. For instance, at $\mu=\mu_{42}$ we have an attracting 8-cycle $N_4$ and all points in
$K(N_2)=[c_2,c_6]\cup[c_8,c_4]\cup[c_4,c_7]\cup[c_5,c_1]$ (and no other point) have a bitrajectory asymptoting
backward to $N_3$.

Fig.~\ref{fig:U2} uses Fig.~\ref{fig:T2} as background, although we omit the points labels to improve
the picture's readability. Over this background, we overlap the curves $c_k=\ell_\mu^k(c)$, $1\leq k\leq6$,
and we shade points of the sets $U_0$ in blue and of the sets $U_1$ in orange throughout the period-3 window.
Finally, we paint in teal the points of the sets $U_2$ in two smaller ranges of parameters, one about
the parameter value $\mu_{51}$ and one about $\mu_{53}$.
Note that we shade points in $U_k$ with the same color as $\cT(N_k)$. We leave $U_{-1}$ in white.

As in case of Fig.~\ref{fig:U1}, we dash in two colors the parts of a trapping region $\cT(N_k)$
that are contained in a $U_{k'}$ with $k'<k$. Moreover, we set to transparent the interior of the line
showing $\cT(N_k)$ when its points belong to $U_{k'}$ with $k'>k$. For instance, in Fig.~\ref{fig:U3}
at $\mu=\mu_{50}$, for all points $x\in[c_2,c_1]\subset J_1(N_0)=[0,1]$ we have that
$s\alpha_{\ell_\mu}(x)=N_0\cup N_1$, where $N_1$ is the red Cantor set, so we paint $J_1(N_0)$ as a hollow
blue line within that range. 

Close to the left endpoint of the period-3 window, the logistic map has three nodes: $N_0$ is the boundary
fixed point, $N_1$ the red Cantor set and $N_2$ is the attracting 3-cycle that arose together with
the repelling 3-cycle at the boundary of the $J_i(N_1)$ (recall that the right endpoint of the period-3
window is the parameter value at which this repelling 3-cycle falls into a chaotic attractor).
As above, all points in $[0,c_1]$ have a bitrajectory asymptoting backward to $N_0$ and all points
in $K(N_0)=[c_2,c_1]$ have a bitrajectory asymptoting backward to $N_1$.

When the attracting 3-cycle bifurcates, we have the same kind of behavior just illustrated above for
Fig.~\ref{fig:U1}. For instance, after the first bifurcation (e.g. see $\mu=\mu_{51}$) all points in
$K(N_1)=[c_2,c_5]\cup[c_3,c_6]\cup[c_4,c_1]$ have a bitrajectory asymptoting backward to the
repelling 3-cycle $N_2$.

Case $\mu_{52}$ is completely analogous to case $\mu_{50}$, the main difference being that
at $\mu_{52}$ the attractor is not a cycle but the trapping region
$$
N_2 = \{[c_2,c_5],[c_3,c_6],[c_4,c_1]\}.
$$
When $\mu=\mu_{53}$, $N_2$ is the blue Cantor set and the attractor $N_3$ is a 9-cycle.
Similarly to the case $\mu=\mu_{51}$, only points of $K(N_1)$ have a bitrajectory asymptoting backward
to the blue Cantor set, while every point in $K(N_0)$ has a bitrajectory asymptoting backward
to the red Cantor set. More details are visible in Fig.~\ref{fig:U3}, where we show a detail
of the central cascade. 

\subsection{Two corollaries.}
Theorem~\ref{thm:main} can be used to prove the Conjecture by Kolyada, Misiurewicz and Snoha
in case of S-unimodal maps as follows.
\begin{cor}
  Let $f:[a,b]\to[a,b]$ be a S-unimodal map. Then $s\alpha_f(x)$ is closed for any $x\in[a,b]$.
\end{cor}
\begin{proof}\
  The only non-trivial case is when the attractor is a Cantor set and $x$ belongs to it.
  In this case, $p=\infty$ and $N_\infty=\cap_{i<\infty}\overline{\Jall(N_i)}$.

  To see this, we first show that $\cap_{i<\infty}\overline{\Jall(N_i)}$ is a node. Indeed,
  let $x$ and $y$ any
  two points in it. Given any $\epsilon>0$, there is some $n_\epsilon$ such that there is some
  component of $U_k$ in both $(x-\epsilon,x+\epsilon)$ and $(y-\epsilon,y+\epsilon)$ for
  every $k>n_\epsilon$. Hence, there is a point $p'$ in the orbit of the periodic point $p_1(N_k)$
  in the first neighborhood and one $p ''$ in the second, so that the sequence
  $\{x,p',f(p'),\dots,f^s(p')=p'',y\}$ is an $\epsilon$-chain from $x$ to $y$.
  For the same reason, $\{y,p'',f(p''),\dots,f^{s'}(p'')=p',x\}$ is an $\epsilon$-chain from $x$ to $y$.
  Moreover,  $x$ and $y$ are both chain-recurrent due to the same argument. Hence $x$ and $y$
  belong to the same node, which can only be $N_\infty$ since there is no other node.
  Since $N_\infty$ is a subset of each $\overline{\Jall(N_i)}$, these are all points of $N_\infty$.
  
  We showed in the previous theorem that, in this case, $s\alpha_f(x)=\cup_{i=0}^\infty N_i$,
  so it is enough to show that the limit points of the set of nodes coincides with the attractor.
  Given any converging sequence $x_k\in N_k$, $k=1,2,\dots$, its limit point cannot be in
  any set $U_i$, with $i$ finite, since, from some $n$ on, every point of the sequence will be out of it. Hence
  it must belong to $N_\infty$.
\end{proof}
The corollary below is a ``$s\alpha$'' version of Theorem~\ref{thm:aR}:
\blue
\begin{cor}
  Let $f:[a,b]\to[a,b]$ be a S-unimodal map. Then $s\alpha_f(c)=J_f\cap\Omega_f$
  if and only if $N_p$ is not a superattracting cycle.
\end{cor}
\begin{proof}\
  Notice, first of all, that all repelling nodes lie in $J_f$ and that the attractor
  lies inside $J_f$ if and only if $N_p$ is of type $A_2$, $A_3$ or $A_5$ and that, in each of
  these three cases, $c\in N_p$.
  Hence, in this case,
  $s\alpha_f(c)=\cup_{k=0}^{p}N_k=J_f\cap\Omega_f$.
  Then notice that either $c\in N_p$ or $c\in U_{p-1}$. Indeed,
  when $N_p$ is of type $A_3$, we have that $c\in N_p$ and, when $N_p$ is not of type
  $A_3$, we have that $c\in\Jall(N_{p-1})\subset K(N_{p-2})$ and $U_{p-1}=K(N_{p-2})\setminus N_p$.
  Hence, 
  when $N_p$ is of type $A_1$ or $A_4$ and the attracting cycle does
  not contain $c$, we must have that $c\in U_{p-1}$ and therefore
  $s\alpha_f(c)=\cup_{k=0}^{p-1}N_k=J_f\cap\Omega_f$.
  Finally, when $N_p$ is superattracting, $s\alpha_f(c)$ contains the attractor but the attractor
  does not belong to $J_f$, so that $s\alpha_f(c)$ is strictly larger than $J_f\cap\Omega_f$.
\end{proof}
\black
%
%
%
%
\section*{Acknowledgments}
The author is in deep debt with Jim Yorke for many useful discussions on the topic
and for many suggestions that helped improving considerably the article after read-proofing
an earlier draft of it. The author is also grateful to Laura Gardini for several discussions
that helped a better understanding of some point of the article, to Mike Boyle for
clarifying some aspects of subshifts of finite type and to Sebastian van Strien for
clarifying the state of the art of backward asymptotics in one-dimensional real dynamics.
Finally, the author is grateful to the anonymous referee for comments that helped fixing
some shortcomings of the original draft.
This work was supported by the National Science Foundation, Grant No. DMS-1832126.
\bibliographystyle{amsplain}
\bibliography{refs}

\providecommand{\bysame}{\leavevmode\hbox to3em{\hrulefill}\thinspace}
\providecommand{\MR}{\relax\ifhmode\unskip\space\fi MR }
\providecommand{\MRhref}[2]{%
  \href{http://www.ams.org/mathscinet-getitem?mr=#1}{#2}
}
\providecommand{\href}[2]{#2}
\begin{thebibliography}{10}

\bibitem{AAN16}
E.~Akin, J.~Auslander, and A.~Nagar, \emph{Variations on the concept of
  topological transitivity}, Studia Mathematica \textbf{235} (2016), 225--249.

\bibitem{BGL13}
F.~Balibrea, J.~Guirao, and M.~Lampart, \emph{A note on the definition of
  $\alpha$-limit set}, Applied Mathematics $\&$ Information Sciences \textbf{7}
  (2013), no.~5, 1929.

\bibitem{Bir20}
G.D. Birkhoff, \emph{Recent advances in dynamics}, Science \textbf{51} (1920),
  no.~1307, 51--55.

\bibitem{Bow75}
R.~Bowen, \emph{$\omega$-limit sets for axiom {A} diffeomorphisms}, Journal of
  differential equations \textbf{18} (1975), no.~2, 333--339.

\bibitem{Boy00}
M.~Boyle, \emph{Algebraic aspects of symbolic dynamics}, Topics in symbolic
  dynamics and applications (Temuco, 1997), London Math. Soc. Lecture Note Ser
  \textbf{279} (2000), 57--88.

\bibitem{CE80}
P.~Collet and J.P. Eckmann, \emph{Iterated maps on the interval as dynamical
  systems}, Springer Science \& Business Media, 1980.

\bibitem{Con78}
C.C. Conley, \emph{Isolated invariant sets and the morse index}, no.~38,
  American Mathematical Soc., 1978.

\bibitem{CD10}
H.~Cui and Y.~Ding, \emph{The $\alpha$-limit sets of a unimodal map without
  homtervals}, Topology and its Applications \textbf{157} (2010), no.~1,
  22--28.

\bibitem{DL18}
R.~{De Leo}, \emph{Conjectures about simple dynamics for some real newton maps
  on $\mathbb{R}^2$}, Fractals \textbf{27} (2019), no.~06, 1950099.

\bibitem{DL20}
R.~De~Leo, \emph{Dynamics of {N}ewton maps of quadratic polynomial maps of
  $\br^2$ into itself}, International Journal of Bifurcation and Chaos
  \textbf{30} (2020), no.~09, 2030027.

\bibitem{DLY20}
R.~De~Leo and J.A. Yorke, \emph{The graph of the logistic map is a tower},
  Discrete and Continuous Dynamical Systems \textbf{41} (2021), no.~11,
  5243--5269.

\bibitem{DLY21}
\bysame, \emph{Infinite towers in the graph of a dynamical system}, Nonlinear
  Dynamics \textbf{105} (2021), 813--835.

\bibitem{dMvS93}
W.~De~Melo and S.~van Strien, \emph{One-dimensional dynamics}, vol.~25,
  Springer Science \& Business Media, 1993.

\bibitem{Fat20b}
Pierre Fatou, \emph{Sur les {\'e}quations fonctionnelles}, Bull. Soc. Math.
  France \textbf{48} (1920), 208--314.

\bibitem{GOY82}
C.~Grebogi, E.~Ott, and J.A. Yorke, \emph{Chaotic attractors in crisis},
  Physical Review Letters \textbf{48} (1982), no.~22, 1507.

\bibitem{Guc79}
J.~Guckenheimer, \emph{Sensitive dependence to initial conditions for one
  dimensional maps}, Communications in Mathematical Physics \textbf{70} (1979),
  no.~2, 133--160.

\bibitem{Guc87}
\bysame, \emph{Limit sets of {S}-unimodal maps with zero entropy},
  Communications in Mathematical Physics \textbf{110} (1987), no.~4, 655--659.

\bibitem{HR20}
J.~Hant{\'a}kov{\'a} and S.~Roth, \emph{On backward attractors of interval
  maps}, arXiv preprint arXiv:2007.10883 (2020).

\bibitem{HK03}
B.~Hasselblatt and A.~Katok, \emph{A first course in dynamics: with a panorama
  of recent developments}, Cambridge University Press, 2003.

\bibitem{HT03}
J.~Hawkins and M.~Taylor, \emph{Maximal entropy measure for rational maps and a
  random iteration algorithm for {Julia} sets}, Intl. J. of Bifurcation and
  Chaos \textbf{13} (2003), no.~6, 1442--1447.

\bibitem{Her92}
M.W. Hero, \emph{Special $\alpha$-limit points for maps of the interval},
  Proceedings of the American Mathematical Society \textbf{116} (1992), no.~4,
  1015--1022.

\bibitem{Kol04}
S.~Kolyada, \emph{Li-{Y}orke sensitivity and other concepts of chaos.},
  Ukrainian Mathematical Journal \textbf{56} (2004), no.~8.

\bibitem{KMS20}
S.~Kolyada, M.~Misiurewicz, and L.~Snoha, \emph{Special $\alpha$-limit sets},
  Contemporary Mathematics \textbf{744} (2020).

\bibitem{KS97}
S.~Kolyada and L.~Snoha, \emph{Some aspects of topological transitivity—a
  survey}, Grazer Math. Ber., Bericht \textbf{334} (1997).

\bibitem{KS09}
\bysame, \emph{{M}inimal dynamical systems}, Scholarpedia \textbf{4} (2009),
  no.~11, 5803, revision \#128110.

\bibitem{LM95}
D.~Lind and B.~Marcus, \emph{An introduction to symbolic dynamics and coding},
  Cambridge university press, 1995.

\bibitem{Lyu02}
M.Yu. Lyubich, \emph{Almost every real quadratic map is either regular or
  stochastic}, Annals of Mathematics \textbf{156} (2002), 1--78.

\bibitem{Mal12}
P.~Mali{\v{c}}k{\`y}, \emph{Backward orbits of transitive maps}, Journal of
  Difference Equations and Applications \textbf{18} (2012), no.~7, 1193--1203.

\bibitem{MMS92}
M.~Martens, W.~De~Melo, and S.~Van~Strien, \emph{Julia-{F}atou-{S}ullivan
  theory for real one-dimensional dynamics}, Acta Mathematica \textbf{168}
  (1992), no.~1, 273--318.

\bibitem{Mil85}
J.W. Milnor, \emph{On the concept of attractor}, The Theory of Chaotic
  Attractors, Springer, 1985, pp.~243--264.

\bibitem{Mit20}
J.~Mitchell, \emph{When is the beginning the end? {O}n full trajectories, limit
  sets and internal chain transitivity}, arXiv preprint arXiv:2004.02878
  (2020).

\bibitem{Nor95b}
D.E. Norton, \emph{The {C}onley decomposition theorem for maps: A metric
  approach}, Rikkyo Daigaku sugaku zasshi \textbf{44} (1995), no.~2, 151--173.

\bibitem{Nor95}
\bysame, \emph{The fundamental theorem of dynamical systems}, Commentationes
  Mathematicae Universitatis Carolinae \textbf{36} (1995), no.~3, 585--597.

\bibitem{Par66}
W.~Parry, \emph{Symbolic dynamics and transformations of the unit interval},
  Transactions of the American Mathematical Society \textbf{122} (1966), no.~2,
  368--378.

\bibitem{Poi90}
H.~Poincar{\'e}, \emph{Sur le probl{\`e}me des trois corps et les {\'e}quations
  de la dynamique}, Acta mathematica \textbf{13} (1890), no.~1, A3--A270.

\bibitem{Poi99}
\bysame, \emph{Les m{\'e}thodes nouvelles de la m{\'e}canique c{\'e}leste},
  vol.~3, Gauthier-Villars et fils, 1899.

\bibitem{Rob08}
C.~Robinson, \emph{What is a chaotic attractor?}, Qualitative Theory of
  Dynamical Systems \textbf{7} (2008), no.~1, 227--236.

\bibitem{Sin78}
D.~Singer, \emph{Stable orbits and bifurcation of maps of the interval}, SIAM
  Journal on Applied Mathematics \textbf{35} (1978), no.~2, 260--267.

\bibitem{Sma67}
S.~Smale, \emph{Differentiable dynamical systems}, Bulletin of the American
  mathematical Society \textbf{73} (1967), no.~6, 747--817.

\bibitem{Sul85}
D.~Sullivan, \emph{Quasiconformal homeomorphisms and dynamics {I. S}olution of
  the {Fatou-Julia} problem on wandering domains}, Annals of mathematics
  \textbf{122} (1985), no.~2, 401--418.

\bibitem{vS81}
S.~van Strien, \emph{On the bifurcations creating horseshoes}, Dynamical
  Systems and Turbulence, Warwick 1980, Springer, 1981, pp.~316--351.

\bibitem{WY73}
F.W. Wilson and J.A. Yorke, \emph{Lyapunov functions and isolating blocks},
  Journal of Differential Equations \textbf{13} (1973), no.~1, 106--123.

\end{thebibliography}
\end{document}